\documentclass[11pt]{article}
\usepackage[a4paper,margin=2.6cm]{geometry}
\usepackage{amsmath,amssymb,amsthm,mathtools,bm}
\usepackage{mathrsfs}
\usepackage{placeins}
\usepackage{booktabs,enumitem,microtype,graphicx,rotating}
\usepackage[numbers,sort&compress]{natbib}
\usepackage[colorlinks=true,linkcolor=blue,citecolor=blue,urlcolor=blue]{hyperref}

\newtheorem{theorem}{Theorem}[section]
\newtheorem{proposition}[theorem]{Proposition}
\newtheorem{lemma}[theorem]{Lemma}

\newtheorem{assumption}[theorem]{Assumption}
\theoremstyle{definition}
\newtheorem{definition}[theorem]{Definition}

\newcommand{\bs}{\bm b_\sigma}

\title{\textbf{Convergence Analysis of a Finite-Volume Scheme for a
		Microglia--Amyloid Chemotaxis Model with Measure-Valued Vascular
		Boundary Sources}}
\author{
\textbf{Elmahdi ERRAJI}\\[2mm]
\small Laboratoire Mathématiques, Informatique et Modélisation des Systèmes Complexes\\
\small École Supérieure de Technologie d'Essaouira, Université Cadi Ayyad\\
\small Km 9, Route d'Agadir, Essaouira Aljadida, BP 383, Essaouira, Morocco\\
\small Corresponding author: \href{mailto:el.erraji@uca.ac.ma}{el.erraji@uca.ac.ma}
}
\date{}

\begin{document}
	\maketitle
	
	\begin{abstract}
		We study a parabolic--parabolic chemotaxis system motivated by microglial recruitment toward an amyloid-$\beta$-associated signal in Alzheimer's disease. The signal is subject to a nonnegative Radon measure-valued Neumann influx on a vascular portion of the boundary, while microglial cells respond to a nonlocal spatial average of the signal. For a fixed sensing length $\sigma>0$, the chemotactic velocity is defined by
		\[
		b_\sigma[v]=\nabla\mathcal K_\sigma[v].
		\]
		For every fixed $\sigma>0$, the nonlocal operator maps finite signal mass into a bounded spatially Lipschitz velocity field. We introduce a weak solution concept adapted to the low regularity induced by the boundary measure. We then construct a fully implicit upwind finite-volume approximation in which the boundary source is discretized through its exact mass on each boundary face--time cell. We establish existence and positivity of the discrete solutions, together with uniform mass, energy, discrete-gradient, and compactness estimates. Finally, we prove that, up to a subsequence, the discrete solutions converge toward a nonnegative weak solution of the continuous problem.
	\end{abstract}
	
	
	\section{Introduction}
	\label{sec:introduction}
	
	Alzheimer's disease is associated with amyloid-$\beta$ accumulation,
	impaired clearance, cerebrovascular dysfunction, and sustained
	neuroinflammation. Soluble amyloid species interact with the brain
	parenchyma and cerebral vasculature through perivascular drainage and
	blood--brain barrier transport
	\cite{Preston2003,TarasoffConway2015,Greenberg2020}. Although vascular
	pathways mainly contribute to clearance, RAGE-mediated transport may
	also carry circulating amyloid-$\beta$ into the brain
	\cite{DeaneEtAl2003}. Microglia migrate toward amyloid-associated
	signals and participate in activation, phagocytosis, and clearance
	responses \cite{RogersLue2001,Cho2013,Lau2023}. These mechanisms
	motivate spatial models coupling microglial migration, soluble
	amyloid-associated signalling, and localized vascular exchange.
	
	Chemotactic models of microglial aggregation around amyloid plaques
	were considered in
	\cite{LucaChavezRossEdelsteinKeshetMogilner2003}, within the broader
	Keller--Segel framework \cite{KellerSegel1970}. Weak solvability and
	finite-volume convergence have also been studied for nonlinear and
	degenerate chemotaxis systems
	\cite{BendahmaneKarlsenUrbano2007,AndreianovBendahmaneSaad2011}.
	Since biological sensing may occur over a finite neighbourhood,
	nonlocal taxis operators provide an alternative to pointwise gradient
	sensing
	\cite{ChenPainterSurulescuZhigun2020,EckardtPainterSurulescuZhigun2020}.
	Related nonlocal aggregation models with degenerate diffusion have been
	investigated analytically and in optimal-control settings
	\cite{BendahmaneKaramiErrajiAtlasAfraites2023}.
	
	We consider an activated microglial density \(u\) and an effective
	soluble amyloid-$\beta$-associated signal \(v\) satisfying
	\[
	\begin{cases}
		\partial_tu-D_u\Delta u
		+\chi\nabla\!\cdot\!\bigl(u\nabla\mathcal K_\sigma[v]\bigr)
		=\mathcal R_u(x,t,u,v),\\
		\partial_tv-D_v\Delta v=\mathcal R_v(x,t,u,v),
	\end{cases}
	\qquad
	\mathcal K_\sigma[v](x)=\displaystyle\int_\Omega
	K_\sigma(x,y)v(y)\,dy .
	\]
	For each fixed sensing length \(\sigma>0\), the assumed kernel
	regularity maps finite signal mass into a bounded Lipschitz
	chemotactic velocity. The reaction terms account for microglial
	activation and loss, signal production and degradation, and
	microglia-mediated clearance.
	
	The distinctive feature is the spatially measure-valued boundary
	influx
	\[
	D_v\partial_\nu v=\mu
	\quad\text{on }\Gamma_{\mathrm v}\times(0,T),
	\qquad
	\partial_\nu v=0
	\quad\text{on }\Gamma_{\mathrm o}\times(0,T),
	\qquad
	\mu\in L^1_{w^*}\bigl(0,T;\mathcal M_+(\Gamma_{\mathrm v})\bigr).
	\]
	Here \(\mu\) represents an effective pathological net inward vascular
	exchange rather than the complete bidirectional blood--brain barrier
	transport. The framework includes concentrated inputs
	\(\mu_t=\sum_jq_j(t)\delta_{\xi_j}\), interpreted as tissue-scale
	idealizations of unresolved vascular patches.
	
	Measure data generally produce substantially lower regularity than
	square-integrable forcing \cite{BoccardoGallouet1989}. Related
	finite-volume theories have been developed for elliptic and Neumann
	problems with low-regularity data
	\cite{DroniouGallouetHerbin2003,AounGuibe2024}. In the present coupled
	system, the boundary measure prevents the standard quadratic estimate
	for \(v\); instead, a global space--time truncation argument yields
	\[
	v\in L^\infty(0,T;L^1(\Omega))
	\cap L^q(0,T;W^{1,q}(\Omega)),
	\qquad
	1<q<\frac{d+2}{d+1}.
	\]
	This subquadratic regularity must still be sufficient to obtain
	compactness and identify the nonlinear chemotactic flux.
	
	Sparse measure controls and inverse-source formulations are commonly
	developed once a suitable forward state relation is available.
	Classical elliptic measure-control theory uses a well-defined
	measure-data state problem
	\cite{CasasClasonKunisch2012}, whereas abstract inverse theories start
	from a prescribed linear forward operator
	\cite{BrediesPikkarainen2013}. For nonlinear chemotaxis systems,
	optimal-control existence can instead be formulated over admissible
	weak state--control triples when uniqueness, and hence a single-valued
	control-to-state map, is unavailable
	\cite{CorreaViannaFilhoGuillenGonzalez2024}. These approaches still
	require existence, compactness, and sequential closedness of the state
	relation. To the best of our knowledge, such foundations have not been
	established for a nonlocal chemotaxis system driven by a
	measure-valued boundary influx. The present work addresses this prior
	analytical gap and prepares the state theory needed for future sparse
	vascular-source identification and control.
	
	We construct a fully implicit finite-volume scheme on admissible
	orthogonal meshes. Two-point fluxes approximate diffusion, an upwind
	flux treats nonlocal taxis, and the boundary datum is represented by
	its exact measure on each boundary face--time cell. This preserves
	local conservation and incorporates singular flux data directly into
	the boundary control volumes
	\cite{EymardGallouetHerbin2000}. Related positivity- and
	structure-preserving Keller--Segel schemes were studied in
	\cite{AndreianovBendahmaneSaad2011,
		BessemoulinChatardJungel2014,AlmeidaBubbaPerthamePouchol2019}.
	
	The main contribution is a finite-volume convergence analysis for this
	nonlocal microglia--amyloid system. For every fixed \(\sigma>0\) and
	every uniformly regular refining sequence of admissible meshes, we
	establish existence and nonnegativity of fully implicit discrete
	solutions without a CFL restriction, derive the required mass, energy,
	truncation, gradient, and translation estimates, and prove
	subsequential convergence toward a nonnegative weak solution. This
	provides an existence and approximation framework, but not yet a
	sparse-control or inverse-problem theory.
	
	Section~\ref{sec:model} introduces the model, assumptions, and weak
	formulation. Section~\ref{sec:finite-volume} presents the scheme and
	its convergence analysis. Section~\ref{sec:numerical} reports the
	numerical experiments, and Section~\ref{sec:conclusion} concludes the
	paper.

	\section{Mathematical model}
	\label{sec:model}
	
	\subsection{Biological model and governing equations}
	\label{subsec:model}
	
	Let \(\Omega\subset\mathbb R^d\), \(d\in\{2,3\}\), and
	\(Q_T:=\Omega\times(0,T)\). We denote by \(u=u(x,t)\) the density of
	motile microglial cells and by \(v=v(x,t)\) an effective soluble
	amyloid-\(\beta\)-associated signal promoting microglial migration
	\cite{Cho2013,Lau2023}.
	
	For a fixed sensing length \(\sigma>0\), define
	\[
	\bs[v]:=\nabla\mathcal K_\sigma[v],
	\qquad
	\mathcal K_\sigma[v](x)
	:=\int_\Omega K_\sigma(x,y)v(y)\,dy.
	\]
	The model is
	\begin{equation}
		\label{eq:ad-nonlocal-model}
		\left\{
		\begin{aligned}
			\partial_tu-D_u\Delta u+\chi\nabla\cdot(u\bs[v])
			&=\mathcal R_u(x,t,u,v)
			&&\text{in }Q_T,\\
			\partial_tv-D_v\Delta v
			&=\mathcal R_v(x,t,u,v)
			&&\text{in }Q_T,\\
			(D_u\nabla u-\chi u\bs[v])\cdot\nu
			&=0
			&&\text{on }\partial\Omega\times(0,T),\\
			D_v\partial_\nu v
			&=\mu
			&&\text{on }\Gamma_{\mathrm v}\times(0,T),\\
			\partial_\nu v
			&=0
			&&\text{on }\Gamma_{\mathrm o}\times(0,T).
		\end{aligned}
		\right.
	\end{equation}
	Here \(D_u,D_v,\chi>0\), and
	\(\mu\in\mathfrak M_T\) is a nonnegative measure-valued vascular
	influx \cite{Preston2003}.
	
	The reaction terms are
	\begin{align}
		\mathcal R_u(x,t,u,v)
		&=
		\lambda_{\mathrm a}
		\frac{v^{m_{\mathrm a}}}
		{K_{\mathrm a}^{m_{\mathrm a}}+v^{m_{\mathrm a}}}
		\bigl(u_{\mathrm r}(x,t)-u\bigr)_+
		-d_uu,
		\label{eq:reaction-u}\\
		\mathcal R_v(x,t,u,v)
		&=
		s_v(x,t)-d_vv
		-\kappa_{\mathrm c}u\frac{v}{K_{\mathrm c}+v},
		\label{eq:reaction-v}
	\end{align}
	where \(r_+:=\max\{r,0\}\). These terms model amyloid-dependent
	microglial activation and loss, distributed signal production and
	degradation, and saturating microglia-mediated clearance.
	
	\subsection{Assumptions on the data}
	\label{subsec:model-data}
	
	\begin{assumption}[Geometry, data, boundary source, and kernel]
		\label{ass:model-data}
		Let \(d\in\{2,3\}\), \(T>0\), and let
		\(\Omega\subset\mathbb R^d\) be a bounded, connected polygonal domain
		if \(d=2\), or polyhedral domain if \(d=3\).
		
		Assume that \(\partial\Omega\) admits a finite conforming decomposition
		\[
		\mathfrak P_{\partial}=\{F_1,\ldots,F_J\},
		\qquad
		\partial\Omega
		=
		\bigcup_{j=1}^J\overline F_j^{\,\partial\Omega},
		\qquad
		F_i\cap F_j=\varnothing\quad(i\neq j),
		\]
		where the \(F_j\) are relatively open planar
		\((d-1)\)-dimensional polytopes and distinct closed cells meet only
		along their relative boundaries. For
		\(\mathfrak P_{\mathrm v}\subset\mathfrak P_{\partial}\), set
		\[
		\Gamma_{\mathrm v}
		:=
		\bigcup_{F\in\mathfrak P_{\mathrm v}}
		\overline F^{\,\partial\Omega},
		\qquad
		\Gamma_{\mathrm o}:=\partial\Omega\setminus\Gamma_{\mathrm v},
		\qquad
		\Sigma:=\partial_{\partial\Omega}\Gamma_{\mathrm v}.
		\]
		Then \(\Gamma_{\mathrm v}\) is compact,
		\(\Sigma\) lies in the \((d-2)\)-dimensional skeleton of
		\(\mathfrak P_{\partial}\), and
		\(\mathcal H^{d-1}(\Sigma)=0\).
		
		Fix \(\sigma>0\). The parameters and distributed data satisfy
		\[
		\begin{gathered}
			D_u,D_v,\chi,\lambda_{\mathrm a},K_{\mathrm a},K_{\mathrm c}>0,
			\qquad
			m_{\mathrm a}\geq1,
			\qquad
			d_u,d_v,\kappa_{\mathrm c}\geq0,\\
			u_{\mathrm r},s_v\in L^\infty_+(Q_T),
			\qquad
			u_0,v_0\in L^\infty_+(\Omega).
		\end{gathered}
		\]
		
		The boundary datum
		\(\mu=(\mu_t)_{t\in(0,T)}\) satisfies
		\[
		\mu\in
		\mathfrak M_T
		:=
		L^1_{w^\ast}
		\bigl(0,T;\mathcal M_+(\Gamma_{\mathrm v})\bigr),
		\qquad
		\mu_t(\Sigma)=0
		\quad\text{for a.e. }t.
		\]
		Its associated space--time measure
		\(\overline\mu\in
		\mathcal M_+(\Gamma_{\mathrm v}\times[0,T])\) is defined by
		\[
		\int_{\Gamma_{\mathrm v}\times[0,T]}\Phi\,d\overline\mu
		:=
		\int_0^T\int_{\Gamma_{\mathrm v}}
		\Phi(\xi,t)\,d\mu_t(\xi)\,dt,
		\qquad
		\Phi\in C(\Gamma_{\mathrm v}\times[0,T]).
		\]
		
		The kernel satisfies
		\[
		K_\sigma\in
		C^1(\overline\Omega\times\overline\Omega)
		\cap W^{2,\infty}(\Omega\times\Omega),
		\qquad
		K_\sigma\geq\kappa_\sigma>0,
		\qquad
		\int_\Omega K_\sigma(x,y)\,dy=1
		\]
		for every \((x,y)\in\overline\Omega\times\overline\Omega\), with the
		last identity understood for every \(x\in\overline\Omega\).
		The reaction terms are those defined in
		\eqref{eq:reaction-u}--\eqref{eq:reaction-v}.
	\end{assumption}
	
	The reactions are Carath\'eodory functions, locally Lipschitz on
	\([0,\infty)^2\), and quasi-positive. Moreover, for a.e.
	\((x,t)\in Q_T\) and all \(u,v\geq0\),
	\[
	\begin{aligned}
		\mathcal R_u(x,t,u,v)
		&\leq\lambda_{\mathrm a}u_{\mathrm r}(x,t)-d_uu,
		&
		|\mathcal R_u(x,t,u,v)|
		&\leq\lambda_{\mathrm a}u_{\mathrm r}(x,t)+d_uu,
		\\
		\mathcal R_v(x,t,u,v)
		&\leq s_v(x,t)-d_vv,
		&
		|\mathcal R_v(x,t,u,v)|
		&\leq s_v(x,t)+d_vv+\kappa_{\mathrm c}u.
	\end{aligned}
	\]
	
	Set
	\[
	\mathscr T_T
	:=
	\left\{
	\zeta\in C^1([0,T];C^2(\overline\Omega)):
	\zeta(\cdot,T)=0
	\right\}.
	\]
	
	\begin{definition}[Weak solution]
		\label{def:weak-solution}
		Let Assumption~\ref{ass:model-data} hold and
		\[
		1<q<q_*:=\frac{d+2}{d+1}.
		\]
		A pair \((u,v)\) is a weak solution of
		\eqref{eq:ad-nonlocal-model} if \(u,v\geq0\) a.e. in \(Q_T\) and
		\[
		\begin{aligned}
			u&\in L^\infty(0,T;L^2(\Omega))
			\cap L^2(0,T;H^1(\Omega)),
			&
			\partial_tu&\in L^2(0,T;H^1(\Omega)'),\\
			v&\in L^\infty(0,T;L^1(\Omega))
			\cap L^q(0,T;W^{1,q}(\Omega)),
		\end{aligned}
		\]
		and, for every \((\varphi,\psi)\in\mathscr T_T^2\),
		\begin{align}
			&-\int_{Q_T}
			\bigl(u\,\partial_t\varphi+v\,\partial_t\psi\bigr)\,dx\,dt
			+\int_{Q_T}
			\bigl(
			D_u\nabla u\cdot\nabla\varphi
			+D_v\nabla v\cdot\nabla\psi
			-\chi u\nabla\mathcal K_\sigma[v]\cdot\nabla\varphi
			\bigr)\,dx\,dt
			\nonumber\\
			&\qquad=
			\int_\Omega
			\bigl(u_0\varphi(\cdot,0)+v_0\psi(\cdot,0)\bigr)\,dx
			+\int_{Q_T}
			\bigl(
			\mathcal R_u(x,t,u,v)\varphi
			+\mathcal R_v(x,t,u,v)\psi
			\bigr)\,dx\,dt
			\nonumber\\
			&\qquad\quad+
			\int_{\Gamma_{\mathrm v}\times(0,T]}
			\psi\,d\overline\mu.
			\label{eq:weak-formulation}
		\end{align}
	\end{definition}
	
	\section{Finite-volume discretization}
	\label{sec:finite-volume}
	
	Throughout this section, every mesh resolves the fixed boundary
	partition \(\mathfrak P_{\partial}\): for each exterior face
	\(e\in\mathcal E_{\mathrm{ext}}\),
	\[
	\operatorname{relint}(e)\subset\Gamma_{\mathrm v}\setminus\Sigma
	\quad\text{or}\quad
	\operatorname{relint}(e)\subset\Gamma_{\mathrm o}.
	\]
	
	\subsection{Discretization framework and scheme}
	\label{subsec:disc-frame}
	
	Let
	\(\mathcal D=(\mathcal T,\mathcal E,\mathcal P)\) be an admissible
	orthogonal finite-volume mesh of \(\Omega\) in the sense of
	\cite{EymardGallouetHerbin2000}, with cells \(K\), centers \(x_K\), and
	\[
	\mathcal E
	=
	\mathcal E_{\mathrm{int}}
	\mathbin{\dot\cup}
	\mathcal E_{\mathrm{ext}}.
	\]
	For \(e=K|L\in\mathcal E_{\mathrm{int}}\), set
	\[
	d_{KL}:=|x_K-x_L|,
	\qquad
	\tau_e:=\frac{|e|}{d_{KL}},
	\qquad
	h_{\mathcal D}:=\max_{K\in\mathcal T}\operatorname{diam}(K),
	\]
	and let \(\nu_{K,e}\) be the unit normal from \(K\) to \(L\).
	
	We consider a uniformly regular family: there exist
	\(\zeta,C_{\mathrm{reg}}>0\), independent of \(\mathcal D\), such that
	\[
	\operatorname{dist}(x_K,e)
	\geq
	\zeta\operatorname{diam}(K)
	\quad
	(e\in\mathcal E_K),
	\]
	and
	\begin{equation}
		\label{eq:mesh-regularity}
		\sum_{\substack{e=K|L\in\mathcal E_{\mathrm{int}}\\
				e\in\mathcal E_K}}
		|e|d_{KL}
		\leq C_{\mathrm{reg}}|K|,
		\qquad K\in\mathcal T.
	\end{equation}
	Constants denoted by \(C\) below are independent of
	\(h_{\mathcal D}\) and \(\Delta t\).
	
	Let
	\[
	t^k:=k\Delta t,
	\qquad
	\Delta t:=\frac{T}{N_T},
	\qquad
	I_k:=(t^k,t^{k+1}],
	\qquad
	\delta_tw_K^{k+1}
	:=
	\frac{w_K^{k+1}-w_K^k}{\Delta t}.
	\]
	
	Let \(X_{\mathcal D}\) be the space of cellwise constant functions.
	For \(1\leq p<\infty\), define
	\[
	\begin{aligned}
		\|w\|_{0,p,\mathcal D}^p
		&:=
		\sum_{K\in\mathcal T}|K||w_K|^p,\\
		|w|_{1,p,\mathcal D}^p
		&:=
		\sum_{e=K|L\in\mathcal E_{\mathrm{int}}}
		|e|d_{KL}
		\left|\frac{w_L-w_K}{d_{KL}}\right|^p,
		\label{eq:discrete-W1p-seminorm}\\
		\|w\|_{1,p,\mathcal D}^p
		&:=
		\|w\|_{0,p,\mathcal D}^p+|w|_{1,p,\mathcal D}^p,
		\qquad
		|w|_{1,\mathcal D}:=|w|_{1,2,\mathcal D}.
	\end{aligned}
	\]
	The discrete dual norm is
	\[
	\|z\|_{-1,\mathcal D}
	:=
	\sup_{\varphi\in X_{\mathcal D}\setminus\{0\}}
	\frac{\sum_K|K|z_K\varphi_K}
	{\|\varphi\|_{1,2,\mathcal D}}.
	\]
	
	For \(e=K|L\in\mathcal E_{\mathrm{int}}\), let
	\[
	D_e
	:=
	\operatorname{conv}(x_K,e)\cup\operatorname{conv}(x_L,e),
	\qquad
	|D_e|=\frac{|e|d_{KL}}d.
	\]
	The diamond gradient is
	\begin{equation}
		\label{eq:discrete-gradient-definition}
		\nabla_{\mathcal D}w
		:=
		d\,\frac{w_L-w_K}{d_{KL}}\nu_{K,e}
		\quad\text{on }D_e,
	\end{equation}
	and is zero on boundary subdiamonds. It satisfies
	\begin{equation}
		\label{eq:discrete-gradient-stability}
		\|\nabla_{\mathcal D}w\|_{L^p(\Omega)}^p
		=
		d^{p-1}|w|_{1,p,\mathcal D}^p.
	\end{equation}
	
	For conservative interior fluxes,
	\(F_{K,e}+F_{L,e}=0\) on \(e=K|L\),
	\begin{equation}
		\label{eq:discrete-integration-by-parts}
		\sum_Kw_K\sum_{e\in\mathcal E_K}F_{K,e}
		=
		\sum_{e=K|L\in\mathcal E_{\mathrm{int}}}
		F_{K,e}(w_K-w_L)
		+
		\sum_K\sum_{e\in\mathcal E_K\cap\mathcal E_{\mathrm{ext}}}
		F_{K,e}w_K.
	\end{equation}
	
	For a sequence \((w^k)_{k=0}^{N_T}\), define, on \(K\times I_k\),
	\[
	w_{\mathcal D}:=w_K^{k+1},
	\qquad
	\delta_tw_{\mathcal D}:=\delta_tw_K^{k+1},
	\qquad
	\nabla_{\mathcal D}w_{\mathcal D}
	:=
	\nabla_{\mathcal D}w^{k+1}.
	\]
	The averaging operators are
	\[
	(\Pi_{\mathcal D}\phi)_K
	:=
	\frac1{|K|}\int_K\phi\,dx,
	\qquad
	(\Pi_{\mathcal D,\Delta t}f)_K^{k+1}
	:=
	\frac1{\Delta t|K|}
	\int_{I_k}\int_Kf\,dx\,dt.
	\]
	
	\begin{lemma}[Discrete functional estimates]
		\label{lem:discrete-functional-estimates}
		Let \(1<p<\infty\). Uniformly with respect to the mesh,
		\[
		\|\Pi_{\mathcal D}\phi\|_{0,p,\mathcal D}
		\leq\|\phi\|_{L^p(\Omega)},
		\qquad
		|\Pi_{\mathcal D}\phi|_{1,p,\mathcal D}
		\leq C\|\nabla\phi\|_{L^p(\Omega)}.
		\]
		Moreover,
		\begin{align}
			\|w\|_{0,p,\mathcal D}
			&\leq
			C\left(
			|w|_{1,p,\mathcal D}
			+
			|\Omega|^{1/p-1}\|w\|_{0,1,\mathcal D}
			\right),
			\label{eq:discrete-Poincare-L1}\\
			\|w\|_{0,2,\mathcal D}^2
			&\leq
			\varepsilon|w|_{1,\mathcal D}^2
			+
			C_\varepsilon\|w\|_{0,1,\mathcal D}^2,
			\qquad \varepsilon>0,
			\label{eq:discrete-GN-absorption}\\
			\|w\|_{0,r_*,\mathcal D}^{r_*}
			&\leq
			C\|w\|_{0,1,\mathcal D}^{2/d}
			\left(
			|w|_{1,\mathcal D}^2
			+
			\|w\|_{0,1,\mathcal D}^2
			\right),
			\qquad
			r_*:=\frac{2(d+1)}d.
			\label{eq:discrete-GN-truncation}
		\end{align}
		Finally, with
		\(\Omega_\xi:=\{x\in\Omega:[x,x+\xi]\subset\Omega\}\),
		\begin{equation}
			\label{eq:discrete-space-translation}
			\int_{\Omega_\xi}
			|w_{\mathcal D}(x+\xi)-w_{\mathcal D}(x)|^p\,dx
			\leq
			C|\xi|(|\xi|+h_{\mathcal D})^{p-1}
			|w|_{1,p,\mathcal D}^p.
		\end{equation}
	\end{lemma}
	
	These estimates are standard on uniformly regular admissible meshes;
	see \cite{EymardGallouetHerbin2000}. Estimate
	\eqref{eq:discrete-GN-absorption} follows from the discrete
	Gagliardo--Nirenberg and Young inequalities.
	
	Set
	\[
	\mathcal E_{\mathrm v}
	:=
	\{e\in\mathcal E_{\mathrm{ext}}:
	\operatorname{relint}(e)\subset\Gamma_{\mathrm v}\},
	\qquad
	\mathcal E_{\mathrm o}
	:=
	\mathcal E_{\mathrm{ext}}\setminus\mathcal E_{\mathrm v},
	\]
	and denote by \(K(e)\) the cell adjacent to \(e\). Choose a Borel
	partition
	\[
	\Gamma_{\mathrm v}\setminus\Sigma
	=
	\biguplus_{e\in\mathcal E_{\mathrm v}}\widehat e,
	\qquad
	\widehat e\subset e\cap\Gamma_{\mathrm v},
	\]
	and define
	\[
	g_e^{k+1}
	:=
	\frac{\overline\mu(\widehat e\times I_k)}
	{\Delta t\,|e|},
	\qquad e\in\mathcal E_{\mathrm v}.
	\]
	Since \(\overline\mu(\Sigma\times[0,T])=0\),
	\begin{equation}
		\label{eq:discrete-boundary-mass-identity}
		\Delta t\sum_{e\in\mathcal E_{\mathrm v}}|e|g_e^{k+1}
		=
		\overline\mu(\Gamma_{\mathrm v}\times I_k),
		\qquad
		\Delta t\sum_{k,e\in\mathcal E_{\mathrm v}}|e|g_e^{k+1}
		=
		\overline\mu(\Gamma_{\mathrm v}\times(0,T]).
	\end{equation}
	
	For \(\alpha\in\{u,v\}\), set
	\[
	\mathcal R_{\alpha,K}^{k+1}(r,s)
	:=
	\frac1{\Delta t|K|}
	\int_{I_k}\int_K
	\mathcal R_\alpha(x,t,r,s)\,dx\,dt.
	\]
	The discrete nonlocal operator is
	\[
	Z_K^\sigma
	:=
	\sum_{M\in\mathcal T}|M|K_\sigma(x_K,x_M),
	\qquad
	\omega_{KM}^\sigma
	:=
	\frac{|M|K_\sigma(x_K,x_M)}{Z_K^\sigma},
	\]
	\[
	K_{\sigma,\mathcal D}[w]_K
	:=
	\sum_{M\in\mathcal T}\omega_{KM}^\sigma w_M,
	\qquad
	c_K^k:=K_{\sigma,\mathcal D}[v^k]_K.
	\]
	Assumption~\ref{ass:model-data} yields
	\[
	Z_K^\sigma\geq\kappa_\sigma|\Omega|,
	\qquad
	\omega_{KM}^\sigma\geq0,
	\qquad
	\sum_M\omega_{KM}^\sigma=1.
	\]
	
	For \(K\in\mathcal T\) and \(k=0,\ldots,N_T-1\), the fully implicit
	scheme is
	\begin{equation}
		\label{eq:finite-volume-scheme}
		\left\{
		\begin{aligned}
			\frac{|K|}{\Delta t}(u_K^{k+1}-u_K^k)
			+\sum_{e\in\mathcal E_K}F_{K,e}^{u,k+1}
			&=
			|K|\mathcal R_{u,K}^{k+1}(u_K^{k+1},v_K^{k+1}),
			\\
			\frac{|K|}{\Delta t}(v_K^{k+1}-v_K^k)
			+\sum_{e\in\mathcal E_K}F_{K,e}^{v,k+1}
			&=
			|K|\mathcal R_{v,K}^{k+1}(u_K^{k+1},v_K^{k+1}).
		\end{aligned}
		\right.
	\end{equation}
	For \(e=K|L\in\mathcal E_{\mathrm{int}}\),
	\[
	\begin{aligned}
		F_{K,e}^{u,k+1}
		&=
		D_u\tau_e(u_K^{k+1}-u_L^{k+1})
		+
		\chi\tau_e
		\Bigl[
		(c_L^{k+1}-c_K^{k+1})^+u_K^{k+1}
		-
		(c_L^{k+1}-c_K^{k+1})^-u_L^{k+1}
		\Bigr],
		\\
		F_{K,e}^{v,k+1}
		&=
		D_v\tau_e(v_K^{k+1}-v_L^{k+1}),
	\end{aligned}
	\]
	where \(r^\pm:=\max\{\pm r,0\}\). On exterior faces,
	\[
	F_{K,e}^{u,k+1}=0,
	\qquad
	F_{K,e}^{v,k+1}
	=
	\begin{cases}
		-|e|g_e^{k+1},&e\in\mathcal E_{\mathrm v},\\
		0,&e\in\mathcal E_{\mathrm o}.
	\end{cases}
	\]
	The initial values are
	\[
	u_K^0=(\Pi_{\mathcal D}u_0)_K,
	\qquad
	v_K^0=(\Pi_{\mathcal D}v_0)_K.
	\]
	
	The interior fluxes are conservative. Moreover, the boundary
	contribution transferred to the right-hand side of the tested signal
	equation is exactly
	\[
	-\Delta t
	\sum_{e\in\mathcal E_{\mathrm v}}
	F_{K(e),e}^{v,k+1}\psi_{K(e)}^{k+1}
	=
	\sum_{e\in\mathcal E_{\mathrm v}}
	\overline\mu(\widehat e\times I_k)
	\psi_{K(e)}^{k+1}.
	\]
	Thus the boundary datum is discretized by its exact mass on each
	face--time cell.
	
	\subsection{Discrete existence and uniform estimates}
	\label{subsec:discrete-estimates}

	The convergence analysis proceeds through discrete solvability,
	mesh-independent estimates, compactness, and flux consistency. These
	ingredients are combined in
	Theorem~\ref{thm:existence-convergence}.
	
	For \(r,s\in\mathbb R\), set \(r^+:=\max\{r,0\}\) and define
	\[
	\begin{aligned}
		\widetilde{\mathcal R}_u(x,t,r,s)
		&:=
		\lambda_{\mathrm a}
		\frac{(s^+)^{m_{\mathrm a}}}
		{K_{\mathrm a}^{m_{\mathrm a}}+(s^+)^{m_{\mathrm a}}}
		\bigl(u_{\mathrm r}(x,t)-r^+\bigr)^+
		-d_ur^+,
		\\
		\widetilde{\mathcal R}_v(x,t,r,s)
		&:=
		s_v(x,t)-d_vs^+
		-\kappa_{\mathrm c}r^+\frac{s^+}{K_{\mathrm c}+s^+}.
	\end{aligned}
	\]
	These continuous extensions coincide with
	\(\mathcal R_u,\mathcal R_v\) on \(\mathbb R_+^2\) and satisfy
	\begin{equation}
		\label{eq:extended-reaction-properties}
		\begin{aligned}
			r<0&\Longrightarrow
			\widetilde{\mathcal R}_u(x,t,r,s)\geq0,
			&
			\widetilde{\mathcal R}_u(x,t,r,s)
			&\leq
			\lambda_{\mathrm a}u_{\mathrm r}(x,t)-d_ur^+,
			\\
			s<0&\Longrightarrow
			\widetilde{\mathcal R}_v(x,t,r,s)\geq0,
			&
			\widetilde{\mathcal R}_v(x,t,r,s)
			&\leq
			s_v(x,t)-d_vs^+.
		\end{aligned}
	\end{equation}
	Their cell--time averages are denoted by
	\(\widetilde{\mathcal R}_{\alpha,K}^{k+1}\),
	\(\alpha\in\{u,v\}\).
	
	\begin{proposition}[Existence and positivity]
		\label{prop:discrete-existence-positivity}
		If \(u_K^0,v_K^0\geq0\) for every \(K\in\mathcal T\), then at every
		time step scheme~\eqref{eq:finite-volume-scheme} admits at least one
		nonnegative solution, without restriction on \(\Delta t\).
	\end{proposition}
	
	\begin{proof}
		Assume inductively that \(u_K^k,v_K^k\geq0\). For
		\(\theta\in[0,1]\), consider
		\begin{equation}
			\label{eq:discrete-existence-homotopy}
			\left\{
			\begin{aligned}
				\frac{|K|}{\Delta t}(U_K-u_K^k)
				+\theta\sum_{e\in\mathcal E_K}F_{K,e}^u(U,V)
				&=
				\theta|K|
				\widetilde{\mathcal R}_{u,K}^{k+1}(U_K,V_K),
				\\
				\frac{|K|}{\Delta t}(V_K-v_K^k)
				+\theta\sum_{e\in\mathcal E_K}F_{K,e}^v(V)
				&=
				\theta|K|
				\widetilde{\mathcal R}_{v,K}^{k+1}(U_K,V_K),
			\end{aligned}
			\right.
		\end{equation}
		where all fluxes, including the vascular boundary flux, are multiplied
		by \(\theta\), and
		\[
		C_K(V):=\sum_{M\in\mathcal T}\omega_{KM}^\sigma V_M
		\]
		is used in the chemotactic flux. Let
		\(\mathcal H_\theta:\mathbb R^{2\#\mathcal T}
		\to\mathbb R^{2\#\mathcal T}\) denote the resulting continuous map.
		
		Every zero \((U,V)\) of \(\mathcal H_\theta\) is nonnegative. Indeed,
		suppose
		\(\mathcal T_U^-:=\{K:U_K<0\}\neq\varnothing\).
		After summing the first equation over \(\mathcal T_U^-\), internal
		fluxes cancel, while for every interface
		\(e=K|L\) with \(K\in\mathcal T_U^-\) and
		\(L\notin\mathcal T_U^-\),
		\[
		F_{K,e}^u(U,V)
		=
		D_u\tau_e(U_K-U_L)
		+
		\chi\tau_e
		\bigl[
		(C_L-C_K)^+U_K-(C_L-C_K)^-U_L
		\bigr]
		\leq0.
		\]
		The time contribution is strictly negative, whereas the right-hand
		side is nonnegative by
		\eqref{eq:extended-reaction-properties}, a contradiction. Thus
		\(U\geq0\). The same argument applied to
		\(\mathcal T_V^-:=\{K:V_K<0\}\), using
		\[
		D_v\tau_e(V_K-V_L)\leq0
		\quad\text{across }\partial\mathcal T_V^-,
		\qquad
		F_{K,e}^v=-|e|g_e^{k+1}\leq0
		\quad\text{on }\Gamma_{\mathrm v},
		\]
		gives \(V\geq0\).
		
		Summing both equations of
		\eqref{eq:discrete-existence-homotopy} over all cells, using flux
		conservativity, the exact boundary-mass identity, and
		\eqref{eq:extended-reaction-properties}, yields
		\begin{equation}
			\label{eq:homotopy-mass-bound}
			\begin{aligned}
				\sum_{K\in\mathcal T}|K|(U_K+V_K)
				\leq{}&
				\sum_{K\in\mathcal T}|K|(u_K^k+v_K^k)
				+\overline\mu(\Gamma_{\mathrm v}\times I_k)
				\\
				&+
				\lambda_{\mathrm a}
				\|u_{\mathrm r}\|_{L^1(\Omega\times I_k)}
				+
				\|s_v\|_{L^1(\Omega\times I_k)}.
			\end{aligned}
		\end{equation}
		Hence the zeros of \(\mathcal H_\theta\) remain in a bounded set,
		uniformly for \(\theta\in[0,1]\).
		
		Choose an open ball \(\mathcal O\) containing all these zeros. At
		\(\theta=0\), the unique zero is \((u^k,v^k)\), and
		\(\mathcal H_0\) has a diagonal Jacobian with positive determinant.
		Therefore,
		\[
		\deg(\mathcal H_0,\mathcal O,0)=1.
		\]
		Homotopy invariance gives
		\[
		\deg(\mathcal H_1,\mathcal O,0)=1,
		\]
		so the extended scheme admits a zero. Since this zero is nonnegative,
		the extended and original reactions coincide. Induction over \(k\)
		completes the proof.
	\end{proof}

	Having established nonnegative discrete solvability, we next derive
	estimates uniform with respect to \(h_{\mathcal D}\) and \(\Delta t\).

\begin{lemma}[Uniform discrete estimates]
	\label{lem:uniform-discrete-estimates}
	Let \((u^k,v^k)_{k=0}^{N_T}\) be a nonnegative solution of
	\eqref{eq:finite-volume-scheme}. For every fixed \(\sigma>0\), there
	exists \(C_{\sigma,T}>0\), independent of
	\(h_{\mathcal D}\) and \(\Delta t\), such that
	\begin{align}
		&\max_{0\leq k\leq N_T}
		\left(
		\|u^k\|_{0,1,\mathcal D}
		+\|v^k\|_{0,1,\mathcal D}
		+\|u^k\|_{0,2,\mathcal D}^2
		\right)
		+
		\sum_{k=0}^{N_T-1}
		\Delta t\,|u^{k+1}|_{1,\mathcal D}^2
		\leq C_{\sigma,T},
		\label{eq:uniform-u-estimates}\\
		&\max_{\substack{0\leq k\leq N_T\\
				e=K|L\in\mathcal E_{\mathrm{int}}}}
		\frac{|c_L^k-c_K^k|}{d_{KL}}
		\leq C_{\sigma,T},
		\label{eq:uniform-c-lipschitz}\\
		&\sum_{k=0}^{N_T-1}
		\Delta t\,
		\|\delta_tu^{k+1}\|_{-1,\mathcal D}^2
		\leq C_{\sigma,T}.
		\label{eq:u-discrete-time-derivative}
	\end{align}
\end{lemma}

\begin{proof}
	Set
	\[
	M_u^k:=\sum_{K\in\mathcal T}|K|u_K^k,
	\qquad
	M_v^k:=\sum_{K\in\mathcal T}|K|v_K^k.
	\]
	Summing the two discrete equations over the cells, using flux
	conservativity and the one-sided reaction bounds, gives
	\[
	\begin{aligned}
		\frac{M_u^{k+1}-M_u^k}{\Delta t}
		+d_uM_u^{k+1}
		&\leq
		\lambda_{\mathrm a}
		\sum_{K\in\mathcal T}|K|u_{{\mathrm r},K}^{k+1},
		\\
		\frac{M_v^{k+1}-M_v^k}{\Delta t}
		+d_vM_v^{k+1}
		&\leq
		\sum_{K\in\mathcal T}|K|s_{v,K}^{k+1}
		+
		\sum_{e\in\mathcal E_{\mathrm v}}|e|g_e^{k+1}.
	\end{aligned}
	\]
	Hence, by \eqref{eq:discrete-boundary-mass-identity},
	\[
	\max_{0\leq k\leq N_T}M_u^k
	\leq
	M_u^0+\lambda_{\mathrm a}\|u_{\mathrm r}\|_{L^1(Q_T)},
	\]
	and
	\[
	\max_{0\leq k\leq N_T}M_v^k
	\leq
	M_v^0+\|s_v\|_{L^1(Q_T)}
	+\overline\mu(\Gamma_{\mathrm v}\times(0,T]).
	\]
	Since the discrete solutions are nonnegative, these estimates give
	the required discrete \(L^1\)-bounds.
	
	The regularity of \(K_\sigma\), together with
	\(Z_K^\sigma\geq\kappa_\sigma|\Omega|\), implies
	\[
	\left|
	\frac{K_\sigma(x_L,x_M)}{Z_L^\sigma}
	-
	\frac{K_\sigma(x_K,x_M)}{Z_K^\sigma}
	\right|
	\leq C_\sigma d_{KL}.
	\]
	Consequently,
	\[
	\begin{aligned}
		|c_L^k-c_K^k|
		&\leq
		C_\sigma d_{KL}
		\sum_{M\in\mathcal T}|M|v_M^k
		\\
		&=
		C_\sigma d_{KL}M_v^k
		\leq
		C_{\sigma,T}d_{KL},
	\end{aligned}
	\]
	which proves \eqref{eq:uniform-c-lipschitz}.
	
	We next test the \(u\)-equation by
	\(\Delta t\,u_K^{k+1}\) and sum over \(K\in\mathcal T\).
	For \(e=K|L\in\mathcal E_{\mathrm{int}}\), introduce the local
	face average
	\[
	\overline u_e^{k+1}
	:=
	\frac{u_K^{k+1}+u_L^{k+1}}{2}.
	\]
	The upwind identity
	\[
	\begin{aligned}
		&(c_L^{k+1}-c_K^{k+1})^+u_K^{k+1}
		-
		(c_L^{k+1}-c_K^{k+1})^-u_L^{k+1}=
		(c_L^{k+1}-c_K^{k+1})\overline u_e^{k+1}
		+
		\frac{|c_L^{k+1}-c_K^{k+1}|}{2}
		\bigl(u_K^{k+1}-u_L^{k+1}\bigr)
	\end{aligned}
	\]
	yields
	\[
	\begin{aligned}
		&\Bigl[
		(c_L^{k+1}-c_K^{k+1})^+u_K^{k+1}
		-
		(c_L^{k+1}-c_K^{k+1})^-u_L^{k+1}
		\Bigr]
		\bigl(u_K^{k+1}-u_L^{k+1}\bigr)
		\\
		&\qquad=
		(c_L^{k+1}-c_K^{k+1})\overline u_e^{k+1}
		\bigl(u_K^{k+1}-u_L^{k+1}\bigr)+
		\frac{|c_L^{k+1}-c_K^{k+1}|}{2}
		\bigl|u_K^{k+1}-u_L^{k+1}\bigr|^2.
	\end{aligned}
	\]
	The last term is nonnegative. Moreover,
	\[
	\begin{aligned}
		&\left|
		\sum_{e=K|L\in\mathcal E_{\mathrm{int}}}
		\tau_e(c_L^{k+1}-c_K^{k+1})
		\overline u_e^{k+1}
		\bigl(u_K^{k+1}-u_L^{k+1}\bigr)
		\right|
		\\
		&\qquad\leq
		C_{\sigma,T}
		\left(
		\sum_{e=K|L\in\mathcal E_{\mathrm{int}}}
		|e|d_{KL}|\overline u_e^{k+1}|^2
		\right)^{1/2}
		|u^{k+1}|_{1,\mathcal D}.
	\end{aligned}
	\]
	Since
	\[
	|\overline u_e^{k+1}|^2
	\leq
	\frac12
	\left(
	|u_K^{k+1}|^2+|u_L^{k+1}|^2
	\right),
	\]
	the mesh regularity \eqref{eq:mesh-regularity} gives
	\[
	\sum_{e=K|L\in\mathcal E_{\mathrm{int}}}
	|e|d_{KL}|\overline u_e^{k+1}|^2
	\leq
	C\|u^{k+1}\|_{0,2,\mathcal D}^2.
	\]
	Therefore,
	\[
	\begin{aligned}
		&\chi
		\left|
		\sum_{e=K|L\in\mathcal E_{\mathrm{int}}}
		\tau_e(c_L^{k+1}-c_K^{k+1})
		\overline u_e^{k+1}
		\bigl(u_K^{k+1}-u_L^{k+1}\bigr)
		\right|
		\\
		&\qquad\leq
		\frac{D_u}{2}|u^{k+1}|_{1,\mathcal D}^2
		+
		C_{\sigma,T}
		\|u^{k+1}\|_{0,2,\mathcal D}^2.
	\end{aligned}
	\]
	
	On the other hand, nonnegativity and the reaction bound imply
	\[
	\begin{aligned}
		&\sum_{K\in\mathcal T}|K|
		\mathcal R_{u,K}^{k+1}
		(u_K^{k+1},v_K^{k+1})u_K^{k+1}
		\\
		&\qquad\leq
		\lambda_{\mathrm a}
		\|u_{\mathrm r}\|_{L^\infty(Q_T)}
		M_u^{k+1}
		\leq C_T.
	\end{aligned}
	\]
	Using also
	\[
	(a-b)a
	\geq
	\frac12(a^2-b^2),
	\]
	discrete integration by parts gives
	\[
	\begin{aligned}
		&\frac12
		\left(
		\|u^{k+1}\|_{0,2,\mathcal D}^2
		-
		\|u^k\|_{0,2,\mathcal D}^2
		\right)
		+
		\frac{D_u\Delta t}{2}
		|u^{k+1}|_{1,\mathcal D}^2
		\\
		&\qquad\leq
		C_{\sigma,T}\Delta t\,
		\|u^{k+1}\|_{0,2,\mathcal D}^2
		+
		C_T\Delta t.
	\end{aligned}
	\]
	Applying \eqref{eq:discrete-GN-absorption} and using the uniform
	\(L^1\)-bound, we choose its parameter sufficiently small to absorb
	the resulting gradient contribution. Thus,
	\[
	\frac12
	\left(
	\|u^{k+1}\|_{0,2,\mathcal D}^2
	-
	\|u^k\|_{0,2,\mathcal D}^2
	\right)
	+
	\frac{D_u\Delta t}{4}
	|u^{k+1}|_{1,\mathcal D}^2
	\leq
	C_{\sigma,T}\Delta t.
	\]
	Summing over \(k\) and using
	\[
	\|u^0\|_{0,2,\mathcal D}
	\leq
	\|u_0\|_{L^2(\Omega)}
	\]
	proves \eqref{eq:uniform-u-estimates}.
	
	Finally, let
	\(\varphi_{\mathcal D}\in X_{\mathcal D}\). Testing the
	\(u\)-equation by \(\varphi_{\mathcal D}\), the diffusion term is
	bounded by
	\[
	D_u|u^{k+1}|_{1,\mathcal D}
	\|\varphi_{\mathcal D}\|_{1,2,\mathcal D}.
	\]
	The chemotactic term is controlled using
	\eqref{eq:uniform-c-lipschitz}, Cauchy--Schwarz, and
	\eqref{eq:mesh-regularity}, while
	\[
	|\mathcal R_u(x,t,u,v)|
	\leq
	\lambda_{\mathrm a}u_{\mathrm r}(x,t)+d_uu
	\]
	controls the reaction term. We therefore obtain
	\[
	\left|
	\left\langle
	\delta_tu^{k+1},\varphi_{\mathcal D}
	\right\rangle_{\mathcal D}
	\right|
	\leq
	C_{\sigma,T}
	\left(
	1+|u^{k+1}|_{1,\mathcal D}
	+\|u^{k+1}\|_{0,2,\mathcal D}
	\right)
	\|\varphi_{\mathcal D}\|_{1,2,\mathcal D}.
	\]
	Taking the supremum over
	\(\varphi_{\mathcal D}\neq0\), squaring, multiplying by
	\(\Delta t\), and summing in \(k\), the preceding estimates give
	\eqref{eq:u-discrete-time-derivative}.
\end{proof}

Lemma~\ref{lem:uniform-discrete-estimates} controls the microglial
component and the nonlocal velocity. Since the measure-valued boundary
datum provides no uniform discrete \(L^2\)-bound for the signal, its
spatial regularity is obtained instead through truncation estimates.
	
	\begin{lemma}[Uniform estimates for the signal]
		\label{lem:uniform-discrete-v-W1p}
		Let \((u^k,v^k)_{k=0}^{N_T}\) be a nonnegative solution of
		\eqref{eq:finite-volume-scheme}. For
		\(1<q<q_*:=(d+2)/(d+1)\), set
		\(q':=q/(q-1)\) and \(Y_q:=W^{1,q'}(\Omega)\).
		Then there exists \(C_{q,T}>0\), independent of
		\(h_{\mathcal D}\) and \(\Delta t\), such that
		\begin{align}
			\sum_{k=0}^{N_T-1}\Delta t\,
			\|v^{k+1}\|_{1,q,\mathcal D}^{q}
			&\leq C_{q,T},
			\label{eq:uniform-discrete-v-full-W1p}\\
			\sum_{k=0}^{N_T-1}\Delta t\,
			\|\delta_tv^{k+1}\|_{Y_q'}
			&\leq C_{q,T}.
			\label{eq:v-time-derivative-Ydual}
		\end{align}
	\end{lemma}
	
\begin{proof}
	By Lemma~\ref{lem:uniform-discrete-estimates},
	\[
	\max_{0\leq k\leq N_T}
	\|v^k\|_{0,1,\mathcal D}
	\leq C_T.
	\]
	Let \(\eta:[0,\infty)\to[0,\infty)\) be nondecreasing,
	\(1\)-Lipschitz, and bounded by \(m>0\), and define
	\[
	H_\eta(s):=\int_0^s\eta(r)\,dr.
	\]
	Testing the signal equation by
	\(\Delta t\,\eta(v_K^{k+1})\), summing over the cells and time
	steps, and using
	\[
	(a-b)\eta(a)\geq H_\eta(a)-H_\eta(b),
	\]
	the nonpositivity of the degradation and clearance terms, and
	\eqref{eq:discrete-boundary-mass-identity}, we obtain
	\begin{equation}
		\label{eq:master-truncation-final}
		\sum_k\Delta t
		\sum_{e=K|L\in\mathcal E_{\mathrm{int}}}
		\tau_e
		|v_L^{k+1}-v_K^{k+1}|
		|\eta(v_L^{k+1})-\eta(v_K^{k+1})|
		\leq Cm.
	\end{equation}
	Indeed, \(0\leq H_\eta(s)\leq ms\), while the initial term,
	the distributed source, and the boundary contribution are
	controlled by the uniform \(L^1\)-bound for \(v\),
	\(\|s_v\|_{L^1(Q_T)}\), and
	\(\overline\mu(\Gamma_{\mathrm v}\times(0,T])\), respectively.
	
	For \(m\geq1\), set
	\[
	T_m(s):=\min\{s,m\},
	\qquad
	r_*:=\frac{2(d+1)}d,
	\qquad
	\beta:=r_*-1=\frac{d+2}{d},
	\]
	and
	\[
	\Lambda_m:=
	\sum_k\Delta t
	\sum_{\{K:\,v_K^{k+1}>m\}}|K|.
	\]
	Choosing \(\eta=T_m\) in
	\eqref{eq:master-truncation-final} gives
	\[
	\sum_k\Delta t\,
	|T_m(v^{k+1})|_{1,\mathcal D}^2
	\leq Cm.
	\]
	Moreover,
	\(\|T_m(v^{k+1})\|_{0,1,\mathcal D}\leq C_T\), and hence
	\eqref{eq:discrete-GN-truncation} yields
	\[
	\begin{aligned}
		\sum_k\Delta t\,
		\|T_m(v^{k+1})\|_{0,r_*,\mathcal D}^{r_*}
		&\leq
		C\sum_k\Delta t
		\left(
		|T_m(v^{k+1})|_{1,\mathcal D}^2+1
		\right)
		\\
		&\leq C(m+1)
		\leq Cm.
	\end{aligned}
	\]
	Since \(T_m(v_K^{k+1})=m\) whenever \(v_K^{k+1}>m\),
	\[
	m^{r_*}\Lambda_m
	\leq
	\sum_k\Delta t\,
	\|T_m(v^{k+1})\|_{0,r_*,\mathcal D}^{r_*}.
	\]
	Consequently,
	\begin{equation}
		\label{eq:truncation-superlevel-estimates}
		\sum_k\Delta t\,
		|T_m(v^{k+1})|_{1,\mathcal D}^2
		\leq Cm,
		\qquad
		\sum_k\Delta t\,
		\|T_m(v^{k+1})\|_{0,r_*,\mathcal D}^{r_*}
		\leq Cm,
		\qquad
		\Lambda_m\leq Cm^{-\beta}.
	\end{equation}
	
	We now use a dyadic decomposition. Define
	\[
	S_{-1}:=T_1,
	\qquad
	S_j(s):=T_{2^j}\bigl((s-2^j)^+\bigr),
	\quad j\geq0.
	\]
	For \(e=K|L\in\mathcal E_{\mathrm{int}}\), set
	\[
	\delta_e^{k+1}
	:=
	|v_L^{k+1}-v_K^{k+1}|,
	\qquad
	\delta_{e,j}^{k+1}
	:=
	|S_j(v_L^{k+1})-S_j(v_K^{k+1})|,
	\]
	and
	\[
	\theta_{e,j}^{k+1}
	:=
	\frac{\delta_{e,j}^{k+1}}{\delta_e^{k+1}}
	\mathbf 1_{\{\delta_e^{k+1}>0\}}.
	\]
	Since
	\[
	s=\sum_{j=-1}^{\infty}S_j(s)
	\qquad (s\geq0),
	\]
	and all \(S_j\) are nondecreasing,
	\[
	\sum_{j=-1}^{\infty}\theta_{e,j}^{k+1}
	=
	\mathbf 1_{\{\delta_e^{k+1}>0\}}.
	\]
	Set
	\[
	E_j:=
	\sum_k\Delta t
	\sum_{e=K|L\in\mathcal E_{\mathrm{int}}}
	\tau_e\delta_e^{k+1}\delta_{e,j}^{k+1},
	\]
	and
	\[
	V_j:=
	\sum_k\Delta t
	\sum_{e=K|L\in\mathcal E_{\mathrm{int}}}
	|e|d_{KL}\theta_{e,j}^{k+1}.
	\]
	
	Since \(S_{-1}\) is bounded by \(1\), while \(S_j\) is bounded
	by \(2^j\) for \(j\geq0\), applying
	\eqref{eq:master-truncation-final} with \(\eta=S_j\) gives
	\[
	E_{-1}\leq C,
	\qquad
	E_j\leq C2^j
	\quad (j\geq0).
	\]
	Furthermore, \(\theta_{e,-1}^{k+1}\leq1\), so the mesh regularity
	\eqref{eq:mesh-regularity} gives
	\[
	V_{-1}
	\leq
	\sum_k\Delta t
	\sum_{e=K|L\in\mathcal E_{\mathrm{int}}}
	|e|d_{KL}
	\leq C.
	\]
	For \(j\geq0\), \(S_j\) is constant on \([0,2^j]\); hence
	\[
	\theta_{e,j}^{k+1}>0
	\quad\Longrightarrow\quad
	\max\{v_K^{k+1},v_L^{k+1}\}>2^j.
	\]
	Therefore,
	\[
	\begin{aligned}
		V_j
		&\leq
		\sum_k\Delta t
		\sum_{e=K|L\in\mathcal E_{\mathrm{int}}}
		|e|d_{KL}
		\left(
		\mathbf 1_{\{v_K^{k+1}>2^j\}}
		+
		\mathbf 1_{\{v_L^{k+1}>2^j\}}
		\right)
		\\
		&\leq
		C\sum_k\Delta t
		\sum_{\{K:\,v_K^{k+1}>2^j\}}|K|
		=
		C\Lambda_{2^j}
		\leq
		C2^{-j\beta},
	\end{aligned}
	\]
	where we used \eqref{eq:mesh-regularity} and
	\eqref{eq:truncation-superlevel-estimates}. Thus,
	\begin{equation}
		\label{eq:dyadic-EV-bounds}
		E_{-1}+V_{-1}\leq C,
		\qquad
		E_j\leq C2^j,
		\qquad
		V_j\leq C2^{-j\beta}
		\quad(j\geq0).
	\end{equation}
	
	Since \(q<q_*<2\), the dyadic partition and Hölder's inequality
	with exponents \(2/q\) and \(2/(2-q)\) give
	\begin{align}
		\sum_k\Delta t\,|v^{k+1}|_{1,q,\mathcal D}^q
		&=
		\sum_{j=-1}^{\infty}
		\sum_k\Delta t
		\sum_{e=K|L\in\mathcal E_{\mathrm{int}}}
		|e|d_{KL}
		\left(
		\frac{\delta_e^{k+1}}{d_{KL}}
		\right)^q
		\theta_{e,j}^{k+1}
		\nonumber\\
		&\leq
		\sum_{j=-1}^{\infty}
		E_j^{q/2}V_j^{1-q/2}.
		\label{eq:dyadic-gradient-estimate}
	\end{align}
	For \(j\geq0\), \eqref{eq:dyadic-EV-bounds} yields
	\[
	E_j^{q/2}V_j^{1-q/2}
	\leq
	C2^{-j\vartheta_q},
	\]
	where
	\[
	\vartheta_q
	:=
	\beta\left(1-\frac q2\right)-\frac q2
	=
	\frac{d+2-(d+1)q}{d}>0.
	\]
	The series is therefore convergent, and
	\[
	\sum_k\Delta t\,
	|v^{k+1}|_{1,q,\mathcal D}^q
	\leq C_{q,T}.
	\]
	The discrete Poincaré inequality
	\eqref{eq:discrete-Poincare-L1}, together with the uniform
	\(L^1\)-bound for \(v\), now gives
	\eqref{eq:uniform-discrete-v-full-W1p}.
	
	It remains to estimate the discrete time derivative. Since
	\[
	q<q_*<\frac{d}{d-1},
	\]
	we have \(q'>d\), and therefore
	\[
	Y_q=W^{1,q'}(\Omega)
	\hookrightarrow C(\overline\Omega).
	\]
	For \(\varphi\in Y_q\), define
	\[
	\left\langle
	\delta_tv^{k+1},\varphi
	\right\rangle
	:=
	\sum_{K\in\mathcal T}
	|K|\delta_tv_K^{k+1}
	(\Pi_{\mathcal D}\varphi)_K,
	\]
	and set
	\[
	G^{k+1}:=
	\sum_{e\in\mathcal E_{\mathrm v}}
	|e|g_e^{k+1}.
	\]
	Testing the signal equation by
	\((\Pi_{\mathcal D}\varphi)_K\), using discrete Hölder
	inequalities, projection stability, and the embedding
	\(Y_q\hookrightarrow C(\overline\Omega)\), gives
	\[
	\begin{aligned}
		\left|
		\left\langle
		\delta_tv^{k+1},\varphi
		\right\rangle
		\right|
		&\leq
		C\left(
		|v^{k+1}|_{1,q,\mathcal D}
		+
		\|\mathcal R_v^{k+1}\|_{0,1,\mathcal D}
		+
		G^{k+1}
		\right)
		\|\varphi\|_{Y_q}.
	\end{aligned}
	\]
	Hence,
	\[
	\|\delta_tv^{k+1}\|_{Y_q'}
	\leq
	C\left(
	|v^{k+1}|_{1,q,\mathcal D}
	+
	\|\mathcal R_v^{k+1}\|_{0,1,\mathcal D}
	+
	G^{k+1}
	\right).
	\]
	Moreover,
	\[
	|\mathcal R_{v,K}^{k+1}|
	\leq
	s_{v,K}^{k+1}
	+d_vv_K^{k+1}
	+\kappa_{\mathrm c}u_K^{k+1}.
	\]
	By Hölder's inequality in time, the gradient estimate already
	obtained, the uniform \(L^1\)-bounds for \(u\) and \(v\), and
	\eqref{eq:discrete-boundary-mass-identity},
	\[
	\sum_k\Delta t
	\left(
	|v^{k+1}|_{1,q,\mathcal D}
	+
	\|\mathcal R_v^{k+1}\|_{0,1,\mathcal D}
	+
	G^{k+1}
	\right)
	\leq C_{q,T}.
	\]
	This proves \eqref{eq:v-time-derivative-Ydual}.
\end{proof}

	The spatial estimates and discrete time-derivative bounds obtained
	above now yield the translation controls required for compactness.
	
	\begin{lemma}[Space and time translation estimates]
		\label{lem:space-time-translates}
		Let
		\[
		1<q<q_*:=\frac{d+2}{d+1},
		\qquad
		q':=\frac{q}{q-1},
		\qquad
		Y_q:=W^{1,q'}(\Omega).
		\]
		Under the estimates of
		Lemmas~\ref{lem:uniform-discrete-estimates} and
		\ref{lem:uniform-discrete-v-W1p}, there exist constants
		\(C_{\sigma,T},C_{q,T}>0\), independent of
		\(h_{\mathcal D}\) and \(\Delta t\), such that, for every
		\(\xi\in\mathbb R^d\),
		\begin{align}
			\int_0^T\int_{\Omega_\xi}
			|u_{\mathcal D}(x+\xi,t)-u_{\mathcal D}(x,t)|^2
			\,dx\,dt
			&\leq
			C_{\sigma,T}|\xi|(|\xi|+h_{\mathcal D}),
			\label{eq:u-space-translate}
			\\
			\int_0^T\int_{\Omega_\xi}
			|v_{\mathcal D}(x+\xi,t)-v_{\mathcal D}(x,t)|^q
			\,dx\,dt
			&\leq
			C_{q,T}|\xi|(|\xi|+h_{\mathcal D})^{q-1}.
			\label{eq:v-space-translate}
		\end{align}
		Moreover, for every \(0<\tau<T\),
		\begin{align}
			\int_0^{T-\tau}
			\|u_{\mathcal D}(\cdot,t+\tau)
			-u_{\mathcal D}(\cdot,t)\|_{L^2(\Omega)}^2\,dt
			&\leq
			C_{\sigma,T}(\tau+\Delta t),
			\label{eq:u-time-translate}
			\\
			\int_0^{T-\tau}
			\|v_{\mathcal D}(\cdot,t+\tau)
			-v_{\mathcal D}(\cdot,t)\|_{Y_q'}\,dt
			&\leq
			C_{q,T}(\tau+\Delta t).
			\label{eq:v-time-translate}
		\end{align}
		Here the \(Y_q'\)-norm is understood through the cell-average dual
		action introduced in
		Lemma~\ref{lem:uniform-discrete-v-W1p}.
	\end{lemma}
	
	\begin{proof}
		The spatial estimates follow from
		\eqref{eq:discrete-space-translation}, applied with \(p=2\) to
		\(u^{k+1}\) and with \(p=q\) to \(v^{k+1}\), followed by multiplication
		by \(\Delta t\), summation in \(k\), and the bounds of
		Lemmas~\ref{lem:uniform-discrete-estimates} and
		\ref{lem:uniform-discrete-v-W1p}.
		
		For the time translations, first let
		\[
		\tau_\ell:=\ell\Delta t,
		\qquad
		1\leq\ell\leq N_T-1.
		\]
		For \(u\), set
		\[
		w_k^\ell
		:=
		u^{k+\ell+1}-u^{k+1}
		=
		\Delta t
		\sum_{j=k+1}^{k+\ell}\delta_tu^{j+1},
		\qquad
		0\leq k\leq N_T-\ell-1.
		\]
		Cauchy--Schwarz, the discrete time-derivative estimate, and a change
		in the order of summation give
		\begin{equation}
			\label{eq:u-shift-negative}
			\sum_{k=0}^{N_T-\ell-1}
			\Delta t\,\|w_k^\ell\|_{-1,\mathcal D}^2
			\leq
			C_{\sigma,T}\tau_\ell^2.
		\end{equation}
		Moreover,
		\[
		\|w_k^\ell\|_{1,\mathcal D}^2
		\leq
		2\|u^{k+\ell+1}\|_{1,\mathcal D}^2
		+
		2\|u^{k+1}\|_{1,\mathcal D}^2,
		\]
		and hence
		\begin{equation}
			\label{eq:u-shift-positive}
			\sum_{k=0}^{N_T-\ell-1}
			\Delta t\,\|w_k^\ell\|_{1,\mathcal D}^2
			\leq
			C_{\sigma,T}.
		\end{equation}
		By the definition of the discrete dual norm,
		\[
		\|w_k^\ell\|_{0,2,\mathcal D}^2
		\leq
		\|w_k^\ell\|_{-1,\mathcal D}
		\|w_k^\ell\|_{1,\mathcal D}.
		\]
		Thus, Cauchy--Schwarz in \(k\), together with
		\eqref{eq:u-shift-negative}--\eqref{eq:u-shift-positive}, yields
		\begin{equation}
			\label{eq:u-discrete-shift}
			\sum_{k=0}^{N_T-\ell-1}
			\Delta t\,
			\|u^{k+\ell+1}-u^{k+1}\|_{0,2,\mathcal D}^2
			\leq
			C_{\sigma,T}\tau_\ell.
		\end{equation}
		
		For \(v\), set
		\[
		z_k^\ell
		:=
		v^{k+\ell+1}-v^{k+1}
		=
		\Delta t
		\sum_{j=k+1}^{k+\ell}\delta_tv^{j+1}.
		\]
		The triangle inequality, a change in the order of summation, and
		\eqref{eq:v-time-derivative-Ydual} give
		\begin{equation}
			\label{eq:v-discrete-shift}
			\sum_{k=0}^{N_T-\ell-1}
			\Delta t\,\|z_k^\ell\|_{Y_q'}
			\leq
			C_{q,T}\tau_\ell.
		\end{equation}
		
		Finally, for \(0<\tau<T\), choose
		\[
		\ell:=\left\lceil\frac{\tau}{\Delta t}\right\rceil,
		\qquad
		\tau\leq\tau_\ell\leq\tau+\Delta t.
		\]
		Since the reconstructions are piecewise constant in time, the indices
		associated with \(t\) and \(t+\tau\) differ by either \(\ell\) or
		\(\ell-1\). Applying
		\eqref{eq:u-discrete-shift} and
		\eqref{eq:v-discrete-shift} to these two shifts proves
		\eqref{eq:u-time-translate} and
		\eqref{eq:v-time-translate}.
	\end{proof}

We now apply the preceding uniform and translation estimates to a
refining sequence of discretizations.

Let \((\mathcal D_m,\Delta t_m)_{m\geq1}\) be a uniformly regular
sequence of admissible discretizations such that
\[
h_m:=h_{\mathcal D_m}\to0,
\qquad
\Delta t_m\to0,
\qquad
u_m:=u_{\mathcal D_m,\Delta t_m},
\qquad
v_m:=v_{\mathcal D_m,\Delta t_m}.
\]
Fix \(1<q<q_*=(d+2)/(d+1)\). The uniform estimates and
Lemma~\ref{lem:space-time-translates} imply that, after extraction,
there exist nonnegative functions
\[
\begin{aligned}
	u&\in L^\infty(0,T;L^2(\Omega))
	\cap L^2(0,T;H^1(\Omega)),
	&
	\partial_tu&\in L^2(0,T;H^1(\Omega)'),\\
	v&\in L^\infty(0,T;L^1(\Omega))
	\cap L^q(0,T;W^{1,q}(\Omega)),
\end{aligned}
\]
such that
\begin{equation}
	\label{eq:compactness-convergences}
	\begin{aligned}
		u_m&\stackrel{*}{\rightharpoonup}u
		&&\text{in }L^\infty(0,T;L^2(\Omega)),
		&
		u_m&\to u
		&&\text{strongly in }L^2(Q_T),\\
		\nabla_{\mathcal D_m}u_m&\rightharpoonup\nabla u
		&&\text{in }L^2(Q_T)^d,
		&
		v_m&\to v
		&&\text{strongly in }L^1(Q_T)\cap L^q(Q_T),\\
		\nabla_{\mathcal D_m}v_m&\rightharpoonup\nabla v
		&&\text{in }L^q(Q_T)^d.
	\end{aligned}
\end{equation}
After a further extraction,
\[
u_m\to u,
\qquad
v_m\to v
\qquad\text{a.e. in }Q_T.
\]

Indeed, the strong compactness of \(u_m\) in \(L^2(Q_T)\) follows
from Lemma~\ref{lem:space-time-translates} and the finite-volume
Fréchet--Kolmogorov criterion. For the signal component, we use the
uniform compatibility estimate
\[
\|w_m\|_{0,1,\mathcal D_m}
\leq
\varepsilon\|w_m\|_{1,q,\mathcal D_m}
+
C_\varepsilon\|w_m\|_{Y_q'},
\qquad
w_m\in X_{\mathcal D_m},
\]
where the last norm is understood through the cell-average dual
action introduced in
Lemma~\ref{lem:uniform-discrete-v-W1p}. This estimate follows by
contradiction from the discrete spatial-translation estimate and the
embedding
\[
Y_q=W^{1,q'}(\Omega)\hookrightarrow C(\overline\Omega).
\]
Indeed, a sequence bounded in the discrete \(W^{1,q}\)-norm is
relatively compact in \(L^1(\Omega)\), whereas convergence to zero in
\(Y_q'\), together with the consistency of
\(\Pi_{\mathcal D_m}\), forces every \(L^1\)-limit to vanish.

Applying the compatibility estimate to
\[
w_m(t):=v_m(t+\tau)-v_m(t),
\]
and using Lemmas~\ref{lem:uniform-discrete-v-W1p} and
\ref{lem:space-time-translates}, we obtain
\[
\begin{aligned}
	\int_0^{T-\tau}
	\|v_m(t+\tau)-v_m(t)\|_{L^1(\Omega)}\,dt
	\leq
	C_{q,T}\varepsilon
	+
	C_\varepsilon(\tau+\Delta t_m).
\end{aligned}
\]
Consequently,
\[
\lim_{\tau\to0}
\limsup_{m\to\infty}
\int_0^{T-\tau}
\|v_m(t+\tau)-v_m(t)\|_{L^1(\Omega)}\,dt
=0.
\]
Together with the spatial-translation estimate in
Lemma~\ref{lem:space-time-translates}, the finite-volume
Fréchet--Kolmogorov criterion yields
\[
v_m\to v
\qquad\text{strongly in }L^1(Q_T).
\]

Choose \(q<r<q_*\). Applying
Lemma~\ref{lem:uniform-discrete-v-W1p} with exponent \(r\) gives a
uniform \(L^r(Q_T)\)-bound. After a further weak extraction in
\(L^r(Q_T)\), the strong \(L^1(Q_T)\)-convergence identifies the weak
limit with \(v\). Interpolation therefore gives
\[
v_m\to v
\qquad\text{strongly in }L^q(Q_T).
\]
The discrete-gradient limits follow from the standard finite-volume
gradient-identification results
\cite{DroniouGallouetHerbin2003,AounGuibe2024}.
Finally, \eqref{eq:u-discrete-time-derivative}, the stability of the
cell-average projection in the discrete \(H^1\)-norm, and discrete
summation by parts identify the weak limit of \(\delta_tu_m\) with
\(\partial_tu\).

The convergences in
\eqref{eq:compactness-convergences} identify the limiting unknowns.
It remains to establish consistency of the discrete diffusion and
complete upwind chemotactic fluxes.

For \(e=K|L\), write
\[
\delta_ez_m^{k+1}:=
z_{m,L}^{k+1}-z_{m,K}^{k+1},
\]
and introduce
\[
\begin{aligned}
	\mathcal A_m(w,z)
	&:=
	\sum_k\Delta t_m
	\sum_{e=K|L\in\mathcal E_{m,\mathrm{int}}}
	\tau_e\,
	\delta_ew_m^{k+1}\delta_ez_m^{k+1},
	\\
	\mathcal C_m(u,c;\varphi)
	&:=
	\sum_k\Delta t_m
	\sum_{e=K|L\in\mathcal E_{m,\mathrm{int}}}
	\tau_e
	\bigl[
	(\delta_ec_m^{k+1})^+u_{m,K}^{k+1}
	-
	(\delta_ec_m^{k+1})^-u_{m,L}^{k+1}
	\bigr]
	(\varphi_{m,K}^{k+1}-\varphi_{m,L}^{k+1}).
\end{aligned}
\]

Now, let us proceed to establish consistency of the discrete fluxes.

\begin{lemma}[Consistency of the discrete fluxes]
	\label{lem:complete-upwind-chemotactic-flux}
	Let \((u_m,v_m)\) satisfy
	\eqref{eq:compactness-convergences}, set
	\[
	c_m^{k+1}:=K_{\sigma,\mathcal D_m}[v_m^{k+1}],
	\qquad
	\varphi_{m,K}^{k+1}:=\varphi(x_K,t_m^{k+1}),
	\qquad
	\psi_{m,K}^{k+1}:=\psi(x_K,t_m^{k+1}),
	\]
	for \(\varphi,\psi\in\mathscr T_T\). Then
	\begin{align}
		\mathcal A_m(u_m,\varphi_m)
		&\longrightarrow
		\int_{Q_T}\nabla u\cdot\nabla\varphi\,dx\,dt,
		\label{eq:u-diffusion-consistency}\\
		\mathcal A_m(v_m,\psi_m)
		&\longrightarrow
		\int_{Q_T}\nabla v\cdot\nabla\psi\,dx\,dt,
		\label{eq:v-diffusion-consistency}\\
		\mathcal C_m(u_m,c_m;\varphi_m)
		&\longrightarrow
		-\int_{Q_T}
		u\,\nabla\mathcal K_\sigma[v]\cdot\nabla\varphi\,dx\,dt.
		\label{eq:chemotactic-flux-consistency}
	\end{align}
\end{lemma}

\begin{proof}
	The diffusion limits follow from
	Proposition~\ref{cor:appendix-diffusion-consistency}, applied with
	\[
	(w_m,\zeta,p)=(u_m,\varphi,2),
	\qquad
	(w_m,\zeta,p)=(v_m,\psi,q).
	\]
	
	For the chemotactic term, set
	\[
	\overline u_{m,e}^{k+1}
	:=
	\frac{u_{m,K}^{k+1}+u_{m,L}^{k+1}}2.
	\]
	The identity
	\[
	a^+r-a^-s
	=
	a\frac{r+s}{2}
	+
	|a|\frac{r-s}{2}
	\]
	gives
	\[
	\mathcal C_m
	=
	\mathcal C_m^{\mathrm{cen}}
	+
	\mathcal R_m^{\mathrm{up}},
	\]
	where
	\[
	\begin{aligned}
		\mathcal C_m^{\mathrm{cen}}
		&:=
		\sum_k\Delta t_m
		\sum_{e=K|L}
		\tau_e\,
		\delta_ec_m^{k+1}\overline u_{m,e}^{k+1}
		(\varphi_{m,K}^{k+1}-\varphi_{m,L}^{k+1}),\\
		\mathcal R_m^{\mathrm{up}}
		&:=
		\frac12\sum_k\Delta t_m
		\sum_{e=K|L}
		\tau_e|\delta_ec_m^{k+1}|
		\delta_eu_m^{k+1}\delta_e\varphi_m^{k+1}.
	\end{aligned}
	\]
	
	Let \(\mathbb M_m,\widetilde c_m,\overline u_m\) be the
	reconstructions of Propositions~\ref{lem:appendix-mesh-tensor} and
	\ref{lem:appendix-nonlocal-potential}. They satisfy
	\[
	\begin{gathered}
		\mathbb M_m\rightharpoonup^\ast I
		\ \text{in }L^\infty(Q_T)^{d\times d},
		\qquad
		\widetilde c_m\to\mathcal K_\sigma[v]
		\ \text{in }L^2(0,T;W^{1,\infty}(\Omega)),\\
		\overline u_m\to u
		\ \text{in }L^2(Q_T),\qquad
		\nabla_{\mathcal D_m}c_m
		=\mathbb M_m\nabla\widetilde c_m+r_m^c,\qquad
		\nabla_{\mathcal D_m}\varphi_m
		=\mathbb M_m\nabla\varphi+r_m^\varphi,\\
		\|r_m^c\|_{L^\infty(Q_T)}
		+\|r_m^\varphi\|_{L^\infty(Q_T)}
		\longrightarrow0.
	\end{gathered}
	\]
	
	Using \(|D_e|=|e|d_{KL}/d\),
	\(\mathbb M_m^\top=\mathbb M_m\), and
	\(\mathbb M_m^2=d\mathbb M_m\), we obtain
	\[
	\begin{aligned}
		\mathcal C_m^{\mathrm{cen}}
		&=
		-\frac1d\int_{Q_T}
		\overline u_m\,
		\nabla_{\mathcal D_m}c_m\cdot
		\nabla_{\mathcal D_m}\varphi_m\,dx\,dt\\
		&=
		-\int_{Q_T}
		\overline u_m\,
		\nabla\widetilde c_m\cdot
		\mathbb M_m\nabla\varphi\,dx\,dt+o(1)\\
		&\longrightarrow
		-\int_{Q_T}
		u\,\nabla\mathcal K_\sigma[v]\cdot\nabla\varphi\,dx\,dt.
	\end{aligned}
	\]
	Indeed,
	\[
	\overline u_m\nabla\widetilde c_m\otimes\nabla\varphi
	\to
	u\nabla\mathcal K_\sigma[v]\otimes\nabla\varphi
	\quad\text{in }L^1(Q_T)^{d\times d},
	\]
	which can be paired with the weak-star convergence of \(\mathbb M_m\).
	
	Finally,
	\[
	|\delta_ec_m^{k+1}|
	\leq C_{\sigma,T}d_{KL},
	\qquad
	|\delta_e\varphi_m^{k+1}|
	\leq C_\varphi d_{KL},
	\]
	and, by Cauchy--Schwarz and mesh regularity,
	\[
	\sum_{e=K|L}|e|d_{KL}|\delta_eu_m^{k+1}|
	\leq
	|u_m^{k+1}|_{1,\mathcal D_m}
	\left(\sum_{e=K|L}|e|d_{KL}^3\right)^{1/2}
	\leq
	Ch_m|u_m^{k+1}|_{1,\mathcal D_m}.
	\]
	Therefore,
	\[
	\begin{aligned}
		|\mathcal R_m^{\mathrm{up}}|
		&\leq
		C_{\sigma,T,\varphi}h_m
		\sum_k\Delta t_m|u_m^{k+1}|_{1,\mathcal D_m}\\
		&\leq
		C_{\sigma,T,\varphi}h_mT^{1/2}
		\left(
		\sum_k\Delta t_m|u_m^{k+1}|_{1,\mathcal D_m}^2
		\right)^{1/2}
		\longrightarrow0.
	\end{aligned}
	\]
	Combining both contributions proves
	\eqref{eq:chemotactic-flux-consistency}.
\end{proof}

Proposition~\ref{prop:discrete-existence-positivity},
the compactness properties in
\eqref{eq:compactness-convergences}, and
Lemma~\ref{lem:complete-upwind-chemotactic-flux} now provide the
ingredients for the main convergence result.
	
	\begin{theorem}[Convergence and existence]
		\label{thm:existence-convergence}
		Let \(\sigma>0\), \(1<q<q_*:=(d+2)/(d+1)\), and let
		Assumption~\ref{ass:model-data} hold. Let
		\((\mathcal D_m,\Delta t_m)_{m\geq1}\) be a uniformly regular sequence
		of admissible discretizations resolving
		\(\mathfrak P_{\partial}\), with
		\(h_{\mathcal D_m}+\Delta t_m\to0\).
		
		For every \(m\), scheme~\eqref{eq:finite-volume-scheme} admits at least
		one nonnegative solution. Moreover, every sequence
		\((u_m,v_m)\) obtained by selecting one such solution for each \(m\)
		contains a subsequence converging as in
		\eqref{eq:compactness-convergences} to a nonnegative weak solution of
		\eqref{eq:ad-nonlocal-model} in the sense of
		Definition~\ref{def:weak-solution}. Consequently, whenever such a
		mesh sequence exists, the continuous problem admits a nonnegative
		weak solution for every fixed \(\sigma>0\).
	\end{theorem}
	
	\begin{proof}
		Proposition~\ref{prop:discrete-existence-positivity} gives a
		nonnegative discrete solution for every \(m\). By the preceding
		compactness results, after extraction,
		\[
		u_m\to u\ \text{in }L^2(Q_T),\qquad
		v_m\to v\ \text{in }L^1(Q_T)\cap L^q(Q_T),\qquad
		(u_m,v_m)\to(u,v)\ \text{a.e. in }Q_T,
		\]
		together with the remaining convergences in
		\eqref{eq:compactness-convergences}. In particular, \(u,v\geq0\)
		almost everywhere.
		
		Fix \((\varphi,\psi)\in\mathscr T_T^2\). For
		\(\zeta\in\{\varphi,\psi\}\), define
		\[
		\zeta_{m,K}^k:=\zeta(x_K,t_m^k),\qquad
		\dot\zeta_m\big|_{K\times I_m^k}
		:=
		\frac{\zeta_{m,K}^{k+1}-\zeta_{m,K}^k}{\Delta t_m},
		\]
		and set, on \(K\times I_m^k\),
		\[
		u_m^-:=u_{m,K}^k,\qquad v_m^-:=v_{m,K}^k.
		\]
		By Proposition~\ref{lem:appendix-left-reconstruction} and the smoothness of
		the test functions,
		\begin{equation}
			\label{eq:left-and-test-convergences}
			u_m^-\to u\ \text{in }L^2(Q_T),\qquad
			v_m^-\to v\ \text{in }L^1(Q_T),\qquad
			(\dot\varphi_m,\dot\psi_m)
			\to(\partial_t\varphi,\partial_t\psi)
			\ \text{uniformly}.
		\end{equation}
		
		Introduce
		\[
		\mathcal T_m(w_m,\zeta_m)
		:=
		\sum_{k,K}|K|
		(w_{m,K}^{k+1}-w_{m,K}^k)\zeta_{m,K}^{k+1}.
		\]
		Since \(\zeta(\cdot,T)=0\), discrete summation by parts gives
		\[
		\mathcal T_m(w_m,\zeta_m)
		=
		-\int_{Q_T}w_m^-\dot\zeta_m\,dx\,dt
		-\sum_K|K|w_{m,K}^0\zeta_{m,K}^0.
		\]
		Moreover,
		\[
		\sum_K|K|w_{m,K}^0\zeta(x_K,0)
		=
		\int_\Omega w_0\,\zeta_m^0\,dx
		\longrightarrow
		\int_\Omega w_0\,\zeta(\cdot,0)\,dx,
		\]
		because \(\|\zeta_m^0-\zeta(\cdot,0)\|_{L^\infty(\Omega)}\to0\).
		Hence \eqref{eq:left-and-test-convergences} yields
		\begin{align}
			\mathcal T_m(u_m,\varphi_m)
			&\longrightarrow
			-\int_{Q_T}u\,\partial_t\varphi\,dx\,dt
			-\int_\Omega u_0\varphi(\cdot,0)\,dx,
			\label{eq:u-time-term-limit}\\
			\mathcal T_m(v_m,\psi_m)
			&\longrightarrow
			-\int_{Q_T}v\,\partial_t\psi\,dx\,dt
			-\int_\Omega v_0\psi(\cdot,0)\,dx.
			\label{eq:v-time-term-limit}
		\end{align}
		
		The two diffusion terms and the complete upwind chemotactic term
		converge by
		Lemma~\ref{lem:complete-upwind-chemotactic-flux}.
		
		For the reactions, set
		\[
		\begin{gathered}
			G_m:=\Pi_m[(u_{\mathrm r}-u_m)^+],\qquad
			G:=(u_{\mathrm r}-u)^+,\qquad
			s_{v,m}:=\Pi_m s_v,\\
			a(s):=\frac{s^{m_{\mathrm a}}}
			{K_{\mathrm a}^{m_{\mathrm a}}+s^{m_{\mathrm a}}},
			\qquad
			h(s):=\frac{s}{K_{\mathrm c}+s}.
		\end{gathered}
		\]
		Proposition~\ref{lem:appendix-cell-time-average} and dominated convergence
		give
		\begin{equation}
			\label{eq:reaction-factor-convergences}
			G_m\to G\ \text{in }L^2(Q_T),\qquad
			s_{v,m}\to s_v\ \text{in }L^1(Q_T),\qquad
			a(v_m)\to a(v),\quad h(v_m)\to h(v)
			\ \text{in }L^r(Q_T)
		\end{equation}
		for every finite \(r\). Therefore,
		\[
		\|a(v_m)G_m-a(v)G\|_{L^1(Q_T)}
		+
		\|u_mh(v_m)-uh(v)\|_{L^1(Q_T)}
		\longrightarrow0,
		\]
		and the discrete reaction reconstructions satisfy
		\begin{equation}
			\label{eq:reaction-limits-main}
			\begin{aligned}
				\mathcal R_{u,m}
				&=
				\lambda_{\mathrm a}a(v_m)G_m-d_uu_m
				\longrightarrow
				\mathcal R_u(\cdot,\cdot,u,v),
				\\
				\mathcal R_{v,m}
				&=
				s_{v,m}-d_vv_m-\kappa_{\mathrm c}u_mh(v_m)
				\longrightarrow
				\mathcal R_v(\cdot,\cdot,u,v)
			\end{aligned}
			\qquad\text{in }L^1(Q_T).
		\end{equation}
		Since the point-sampled test reconstructions converge uniformly,
		\eqref{eq:reaction-limits-main} permits passage to the limit in both
		reaction terms.
		
		Finally, Proposition~\ref{lem:appendix-boundary-measure} gives
		\[
		\sum_{k}\sum_{e\in\mathcal E_{m,\mathrm v}}
		\overline\mu(\widehat e\times I_m^k)
		\psi(x_{K(e)},t_m^{k+1})
		\longrightarrow
		\int_{\Gamma_{\mathrm v}\times(0,T]}\psi\,d\overline\mu
		=
		\int_0^T\int_{\Gamma_{\mathrm v}}\psi\,d\mu_t\,dt.
		\]
		Since \(F_{K,e}^{v,k+1}=-|e|g_e^{k+1}\) on
		\(\mathcal E_{m,\mathrm v}\), this contribution appears with a positive
		sign after the exterior flux term is transferred to the right-hand
		side.
		
		Passing to the limit in the tested scheme using
		\eqref{eq:u-time-term-limit}--\eqref{eq:v-time-term-limit},
		Lemma~\ref{lem:complete-upwind-chemotactic-flux},
		\eqref{eq:reaction-limits-main}, and the boundary limit gives
		\eqref{eq:weak-formulation}. Hence \((u,v)\) is a nonnegative weak
		solution.
	\end{proof}
	
\section{Numerical consistency tests}
\label{sec:numerical}

\subsection{Implementation}
\label{subsec:numerical-implementation}

All computations use cell-centred two-point finite volumes on clipped
Voronoi meshes of \(\Omega=(0,1)^2\). The implementation checks positive
transmissibilities, unique face ownership, unit total area, and
generator--face orthogonality. We write \(N=\#\mathcal T\) and report
\[
h_{\mathcal D}
=\max_{K\in\mathcal T}\operatorname{diam}(K).
\]
These checks document the meshes used in the computations; no numerical
claim concerning mesh-uniform regularity constants is made.

For an interior face \(e=K|L\), the outgoing numerical fluxes are
\[
F^u_{K,e}
=D_u\tau_e(u_K-u_L)
+\chi\tau_e
\left[
(c_L-c_K)^+u_K-(c_L-c_K)^-u_L
\right],
\qquad
F^v_{K,e}=D_v\tau_e(v_K-v_L).
\]
Each interior contribution is assembled once and added to the two adjacent
cells with opposite signs. No boundary face enters the microglial flux, and
no clipping, artificial diffusion, or positivity correction is applied.

The vascular rate is
\[
q(t)
=Q_0\sin^2\!\left(
\pi\frac{t-t_{\rm on}}{L}
\right)
\mathbf 1_{[t_{\rm on},t_{\rm off}]}(t),
\qquad
L=t_{\rm off}-t_{\rm on},
\]
with
\[
q'(t)
=
\begin{cases}
	\dfrac{\pi Q_0}{L}
	\sin\!\left(
	2\pi\dfrac{t-t_{\rm on}}{L}
	\right),
	& t_{\rm on}<t<t_{\rm off},\\[1ex]
	0,
	& \text{otherwise}.
\end{cases}
\]
For
\[
a=\max(t_n,t_{\rm on}),
\qquad
b=\min(t_{n+1},t_{\rm off}),
\]
the exact pulse mass on \(I_n=(t_n,t_{n+1}]\) is \(Q_n=0\) if \(b\leq a\),
and otherwise
\[
Q_n
=Q_0\left[
\frac{b-a}{2}
-\frac{L}{4\pi}
\left(
\sin\frac{2\pi(b-t_{\rm on})}{L}
-\sin\frac{2\pi(a-t_{\rm on})}{L}
\right)
\right].
\]
In the manufactured experiment, each face
\(e\subset\Gamma_{\rm v}=\{1\}\times(0,1)\) receives the mass
\(|e|Q_n\). Since \(\mathcal H^1(\Gamma_{\rm v})=1\), the total assembled
boundary mass equals \(Q_n\) up to roundoff.

The bounded-domain Gaussian kernel is
\[
G_\sigma(X,Y)
=\exp\!\left(-\frac{|X-Y|^2}{2\sigma^2}\right),
\qquad
K_\sigma(X,Y)
=\frac{G_\sigma(X,Y)}{Z_\sigma(X)},
\]
where
\[
Z_\sigma(X)
=\prod_{j=1}^2
\sigma\sqrt{\frac{\pi}{2}}
\left[
\operatorname{erf}\!\left(
\frac{1-X_j}{\sqrt2\sigma}
\right)
+
\operatorname{erf}\!\left(
\frac{X_j}{\sqrt2\sigma}
\right)
\right].
\]
The discrete convolution is
\[
c_K=\sum_M\omega^\sigma_{KM}v_M,
\qquad
\omega^\sigma_{KM}
=
\frac{|M|G_\sigma(x_K,x_M)}
{\sum_J|J|G_\sigma(x_K,x_J)}.
\]
All weights are nonnegative, and their row sums are checked to within
\(10^{-12}\).

Backward Euler is used in time. At each step, a Picard iteration updates the
reaction-linearized signal equation, the discrete convolution, and the
reaction-linearized upwind microglial equation. Linear systems are solved
directly, and a failed nonlinear or linear solve aborts the computation.

For \(z\in\{u,v\}\), let
\[
\mathscr F_{z,K}^{n+1}
=
|K|\bigl(z_K^{n+1}-z_K^n\bigr)
+\Delta t\sum_{e\in\mathcal E_K}F_{K,e}^{z,n+1}
-\Delta t\,|K|
\left(
R_{z,K}^{n+1}+f_{z,K}^{n+1}
\right)
\]
denote the residual of the time-integrated conservative cell balance. For
the signal equation, the exterior vascular flux is included in the face
sum. The normalized nonlinear residual is
\[
{\rm Res}^{n+1}
=
\max_{z\in\{u,v\}}
\frac{
	\displaystyle
	\sum_{K\in\mathcal T}
	\left|\mathscr F_{z,K}^{n+1}\right|
}{
	\displaystyle
	\max\left\{
	1,\,
	\sum_K|K||z_K^{n+1}|,\,
	\sum_K|K||z_K^n|
	\right\}
}.
\]
Thus both the numerator and denominator have the scale of an integrated
cell mass. The full residual is evaluated after the Picard iteration, and
all reported time steps satisfy
\[
{\rm Res}^{n+1}\leq 10^{-10}.
\]

The one-step mass-balance defect is
\[
\mathcal B_z^{n+1}
=M_z^{n+1}-M_z^n
-\Delta t\sum_K|K|
\left(
R_{z,K}^{n+1}+f_{z,K}^{n+1}
\right)
-\delta_{zv}Q_n.
\]

\subsection{Exact manufactured-solution test}
\label{subsec:manufactured}

The test uses the absolutely continuous boundary measure
\[
\mu_t
=q(t)\mathcal H^1\!\restriction_{\Gamma_{\rm v}},
\qquad
\Gamma_{\rm v}=\{1\}\times(0,1).
\]
This datum belongs to the measure class considered analytically. The
experiment does not provide a numerical benchmark for a spatially atomic
boundary source.

Writing \(X=(x_1,x_2)\), define
\[
v_{\rm ex}(X,t)
=v_b+\frac{q(t)x_1^2}{2D_v}.
\]
Then
\[
\partial_tv_{\rm ex}
=\frac{q'(t)x_1^2}{2D_v},
\qquad
\nabla v_{\rm ex}
=\left(
\frac{q(t)x_1}{D_v},0
\right),
\qquad
\Delta v_{\rm ex}
=\frac{q(t)}{D_v},
\]
and
\[
D_v\partial_\nu v_{\rm ex}=q(t)
\quad\text{on }\Gamma_{\rm v},
\qquad
\partial_\nu v_{\rm ex}=0
\quad\text{on }\partial\Omega\setminus\Gamma_{\rm v}.
\]

For the exact sensed field, define
\[
E_0(x)=e^{-x^2/(2\sigma^2)},
\qquad
E_1(x)=e^{-(1-x)^2/(2\sigma^2)},
\]
\[
I_j(x)
=\int_0^1s^j e^{-(x-s)^2/(2\sigma^2)}\,ds,
\qquad
R_\sigma(x)=\frac{I_2(x)}{I_0(x)}.
\]
The required expressions are
\[
I_0(x)
=\sigma\sqrt{\frac{\pi}{2}}
\left[
\operatorname{erf}\!\left(
\frac{1-x}{\sqrt2\sigma}
\right)
+
\operatorname{erf}\!\left(
\frac{x}{\sqrt2\sigma}
\right)
\right],
\]
\[
I_2(x)
=(x^2+\sigma^2)I_0(x)
+\sigma^2xE_0(x)
-\sigma^2(1+x)E_1(x),
\]
\[
I_0'(x)=E_0(x)-E_1(x),
\qquad
I_2'(x)
=2xI_0(x)+2\sigma^2E_0(x)
-(1+2\sigma^2)E_1(x),
\]
and
\[
R_\sigma'(x)
=
\frac{
	I_2'(x)I_0(x)-I_2(x)I_0'(x)
}{I_0(x)^2}.
\]
The Gaussian factor in the second coordinate cancels in the normalized
convolution, so
\[
c_{\rm ex}(X,t)
=v_b+\frac{q(t)}{2D_v}R_\sigma(x_1),
\]
\[
\partial_tc_{\rm ex}(X,t)
=\frac{q'(t)}{2D_v}R_\sigma(x_1),
\qquad
\nabla c_{\rm ex}(X,t)
=\frac{q(t)}{2D_v}
\left(R_\sigma'(x_1),0\right).
\]
These expressions are used for the exact reference and manufactured
forcing. The finite-volume computation retains the row-normalized discrete
convolution.

Set
\[
\gamma=\frac{\chi}{D_u},
\qquad
w(X,t)
=1+\varepsilon e^{-\omega t}
\cos(\pi x_1)\cos(\pi x_2),
\]
and
\[
u_{\rm ex}(X,t)
=e^{\gamma c_{\rm ex}(X,t)}w(X,t).
\]
The required derivatives are
\[
w_t
=-\omega\varepsilon e^{-\omega t}
\cos(\pi x_1)\cos(\pi x_2),
\]
\[
\nabla w
=-\pi\varepsilon e^{-\omega t}
\begin{pmatrix}
	\sin(\pi x_1)\cos(\pi x_2)\\
	\cos(\pi x_1)\sin(\pi x_2)
\end{pmatrix},
\qquad
\Delta w
=-2\pi^2\varepsilon e^{-\omega t}
\cos(\pi x_1)\cos(\pi x_2).
\]
Since \(\gamma=\chi/D_u\),
\[
D_u\nabla u_{\rm ex}
-\chi u_{\rm ex}\nabla c_{\rm ex}
=D_ue^{\gamma c_{\rm ex}}\nabla w,
\]
whose normal component vanishes on \(\partial\Omega\).

The physical reaction terms remain active. The manufactured sources are
\[
f_u
=e^{\gamma c_{\rm ex}}
\left[
w_t+\gamma w\,\partial_tc_{\rm ex}
-D_u\Delta w
-\chi\nabla w\cdot\nabla c_{\rm ex}
\right]
-R_u(X,t,u_{\rm ex},v_{\rm ex}),
\]
and
\[
f_v
=\frac{q'(t)x_1^2}{2D_v}
-q(t)-s_v(X,t)+d_vv_{\rm ex}
+\frac{\kappa_cu_{\rm ex}v_{\rm ex}}
{K_c+v_{\rm ex}}.
\]
Direct substitution verifies the two augmented equations and their boundary
conditions. Since the added sources need not be quasi-positive,
nonnegativity in this experiment is reported only as a numerical
observation.

The parameters are
\[
\begin{gathered}
	D_u=0.01,\quad D_v=0.05,\quad \chi=0.002,\quad
	d_v=0.15,\quad \sigma=0.08,\quad T=0.4,\\
	\varepsilon=0.05,\quad \omega=0.1,\quad v_b=0.05,\quad
	Q_0=0.08,\quad t_{\rm on}=0.01,\quad t_{\rm off}=0.39,\\
	\lambda_a=0.35,\quad u_r=1.4,\quad d_u=0.04,\quad
	m_a=2,\quad K_a=0.2,\quad s_v=0,\quad
	\kappa_c=0.15,\quad K_c=0.1.
\end{gathered}
\]
Picard iteration is capped at \(40\) iterations and uses no relaxation.
Initial values are polygon averages of the exact initial data.

The manufactured sources are evaluated as cell--time averages, with time
slabs split at \(t_{\rm on}\) and \(t_{\rm off}\). Source averages use an
eight-point polygon Duffy rule and a six-point time rule. Reference cell
averages use a 12-point polygon rule, while the error norms use a ten-point
polygon rule and a four-point time rule. These quadrature rules are kept
fixed throughout each refinement sequence.

For \(p\in\{1,2\}\), and also \(p=q=1.25\) for \(v\), define
\[
\bar z_K(t)
=\frac{1}{|K|}\int_Kz_{\rm ex}(X,t)\,dX,
\]
and
\[
E_{p,p}(z)^p
=\sum_n\int_{I_n}
\sum_K|K|
\left|
z_K^{n+1}-\bar z_K(t)
\right|^p\,dt.
\]
The finite-volume reconstruction is right-continuous and constant on every
time slab.

The first sequence uses five perturbed non-Cartesian Voronoi meshes with
\[
\Delta t\simeq0.005h_{\mathcal D}^2.
\]
For successive discretizations, the reported path EOC is
\[
{\rm EOC}_{\rm path}
=
\frac{\log(E_i/E_{i+1})}
{\log(h_{\mathcal D_i}/h_{\mathcal D_{i+1}})}.
\]
Because \(h_{\mathcal D}\) and \(\Delta t\) decrease simultaneously, these
values are descriptive slopes along the selected joint refinement path.
They are not interpreted as a separate asymptotic spatial order.

The temporal study uses a fixed mesh with \(N=100\) control volumes. Its
fine-time reference uses 1280 time steps and is checked against a 2560-step
trajectory on the common time partition. The temporal EOC is
\[
{\rm EOC}_{t}
=
\frac{\log(E_{\Delta t}/E_{\Delta t/2})}{\log 2}.
\]

Here \(N_{\Gamma_{\rm v}}\) in Table~\ref{tab:ms-spatial} denotes the number
of boundary faces contained in \(\Gamma_{\rm v}\). The average Picard count
is taken over all completed time steps.

\FloatBarrier
\begin{table}[h]
	\centering
	\scriptsize
	\caption{Errors along the joint refinement path
		\(\Delta t\simeq0.005h_{\mathcal D}^2\) on perturbed Voronoi meshes.
		Here \(N=\#\mathcal T\), \(q=1.25\), and
		\(N_{\Gamma_{\rm v}}\) is the number of vascular boundary faces.
		Each error entry is error/path EOC. The path EOCs are descriptive
		finite-resolution slopes and are not interpreted as separate spatial
		convergence orders. The average Picard count is taken over all completed
		time steps.}
	\label{tab:ms-spatial}
	
	\resizebox{\linewidth}{!}{%
		\begin{tabular}{@{}rrrrrrrr@{}}
			\toprule
			\(N\)
			& \(h_{\mathcal D}\)
			& \(\Delta t\)
			& \(E_{1,1}(u)\)
			& \(E_{2,2}(u)\)
			& \(E_{1,1}(v)\)
			& \(E_{2,2}(v)\)
			& \(E_{q,q}(v)\)
			\\
			\midrule
			16
			& \(3.754\times10^{-1}\)
			& \(7.042\times10^{-4}\)
			& \(7.15\times10^{-5}\;/\;\text{--}\)
			& \(1.56\times10^{-4}\;/\;\text{--}\)
			& \(6.03\times10^{-4}\;/\;\text{--}\)
			& \(1.42\times10^{-3}\;/\;\text{--}\)
			& \(8.18\times10^{-4}\;/\;\text{--}\)
			\\
			25
			& \(2.941\times10^{-1}\)
			& \(4.324\times10^{-4}\)
			& \(5.05\times10^{-5}\;/\;1.42\)
			& \(1.12\times10^{-4}\;/\;1.35\)
			& \(3.81\times10^{-4}\;/\;1.87\)
			& \(1.06\times10^{-3}\;/\;1.19\)
			& \(5.45\times10^{-4}\;/\;1.66\)
			\\
			36
			& \(2.493\times10^{-1}\)
			& \(3.108\times10^{-4}\)
			& \(3.96\times10^{-5}\;/\;1.47\)
			& \(8.87\times10^{-5}\;/\;1.42\)
			& \(3.00\times10^{-4}\;/\;1.45\)
			& \(8.58\times10^{-4}\;/\;1.28\)
			& \(4.31\times10^{-4}\;/\;1.43\)
			\\
			49
			& \(2.150\times10^{-1}\)
			& \(2.311\times10^{-4}\)
			& \(3.03\times10^{-5}\;/\;1.82\)
			& \(7.01\times10^{-5}\;/\;1.58\)
			& \(2.26\times10^{-4}\;/\;1.93\)
			& \(7.23\times10^{-4}\;/\;1.16\)
			& \(3.36\times10^{-4}\;/\;1.67\)
			\\
			64
			& \(1.883\times10^{-1}\)
			& \(1.773\times10^{-4}\)
			& \(2.49\times10^{-5}\;/\;1.49\)
			& \(6.05\times10^{-5}\;/\;1.11\)
			& \(1.87\times10^{-4}\;/\;1.43\)
			& \(6.26\times10^{-4}\;/\;1.08\)
			& \(2.82\times10^{-4}\;/\;1.32\)
			\\
			\bottomrule
		\end{tabular}%
	}
	
	\vspace{1mm}
	
	\resizebox{\linewidth}{!}{%
		\begin{tabular}{@{}rrrrrrrr@{}}
			\toprule
			\(N\)
			& \(\min_{n,K}u_K^n\)
			& \(\min_{n,K}v_K^n\)
			& \(\max_n {\rm Res}^{n+1}\)
			& \(\max_n|\mathcal B_u^{n+1}|\)
			& \(\max_n|\mathcal B_v^{n+1}|\)
			& Picard it.\ max/avg
			& \(N_{\Gamma_{\rm v}}\)
			\\
			\midrule
			16
			& \(9.691\times10^{-1}\)
			& \(4.522\times10^{-2}\)
			& \(7.99\times10^{-12}\)
			& \(5.10\times10^{-16}\)
			& \(5.59\times10^{-15}\)
			& \(4/3.73\)
			& 4
			\\
			25
			& \(9.657\times10^{-1}\)
			& \(4.467\times10^{-2}\)
			& \(5.39\times10^{-12}\)
			& \(4.47\times10^{-16}\)
			& \(2.36\times10^{-15}\)
			& \(4/3.39\)
			& 5
			\\
			36
			& \(9.638\times10^{-1}\)
			& \(4.531\times10^{-2}\)
			& \(2.54\times10^{-12}\)
			& \(4.50\times10^{-16}\)
			& \(8.05\times10^{-16}\)
			& \(3/3.00\)
			& 6
			\\
			49
			& \(9.627\times10^{-1}\)
			& \(4.614\times10^{-2}\)
			& \(1.06\times10^{-12}\)
			& \(5.16\times10^{-16}\)
			& \(2.63\times10^{-16}\)
			& \(3/3.00\)
			& 7
			\\
			64
			& \(9.620\times10^{-1}\)
			& \(4.650\times10^{-2}\)
			& \(1.29\times10^{-12}\)
			& \(4.67\times10^{-16}\)
			& \(1.25\times10^{-16}\)
			& \(3/2.98\)
			& 8
			\\
			\bottomrule
		\end{tabular}%
	}
\end{table}

\begin{table}[h]
	\centering
	\scriptsize
	\caption{Fixed-mesh temporal errors against the checked fine-time reference
		on the \(N=100\) control-volume mesh. Each error entry is error/EOC.
		Solver and balance diagnostics are maxima over all completed time steps.}
	\label{tab:ms-temporal}
	
	\resizebox{\linewidth}{!}{%
		\begin{tabular}{@{}rrrrrrrrrr@{}}
			\toprule
			\(\Delta t\)
			& \(E_{1,1}(u)\)
			& \(E_{2,2}(u)\)
			& \(E_{1,1}(v)\)
			& \(E_{2,2}(v)\)
			& \(E_{q,q}(v)\)
			& \(\max_n {\rm Res}^{n+1}\)
			& Picard it.\ max/avg
			& \(\max_n|\mathcal B_u^{n+1}|\)
			& \(\max_n|\mathcal B_v^{n+1}|\)
			\\
			\midrule
			\(1.000\times10^{-1}\)
			& \(5.91\times10^{-3}\;/\;\text{--}\)
			& \(1.47\times10^{-2}\;/\;\text{--}\)
			& \(2.62\times10^{-2}\;/\;\text{--}\)
			& \(6.71\times10^{-2}\;/\;\text{--}\)
			& \(3.68\times10^{-2}\;/\;\text{--}\)
			& \(1.98\times10^{-12}\)
			& \(8/8.00\)
			& \(1.35\times10^{-16}\)
			& \(1.98\times10^{-13}\)
			\\
			\(5.000\times10^{-2}\)
			& \(2.89\times10^{-3}\;/\;1.03\)
			& \(7.68\times10^{-3}\;/\;0.93\)
			& \(1.32\times10^{-2}\;/\;0.99\)
			& \(3.59\times10^{-2}\;/\;0.90\)
			& \(1.88\times10^{-2}\;/\;0.97\)
			& \(5.41\times10^{-12}\)
			& \(7/6.75\)
			& \(1.96\times10^{-16}\)
			& \(2.70\times10^{-13}\)
			\\
			\(2.500\times10^{-2}\)
			& \(1.42\times10^{-3}\;/\;1.03\)
			& \(3.85\times10^{-3}\;/\;1.00\)
			& \(6.59\times10^{-3}\;/\;1.00\)
			& \(1.82\times10^{-2}\;/\;0.98\)
			& \(9.45\times10^{-3}\;/\;0.99\)
			& \(5.47\times10^{-12}\)
			& \(6/5.88\)
			& \(4.41\times10^{-16}\)
			& \(1.37\times10^{-13}\)
			\\
			\bottomrule
		\end{tabular}%
	}
\end{table}
\FloatBarrier

All reported errors decrease along the joint refinement path. On the fixed
mesh, the temporal EOCs range from \(0.90\) to \(1.03\). This is compatible
with first-order backward Euler over the tested finite range, but is not
presented as an independent convergence theorem. The difference between the
1280- and 2560-step reference trajectories is \(0.65\)--\(0.78\%\) of the
smallest reported temporal error. Across the reported computations, the
largest nonlinear residual is \(7.99\times10^{-12}\), the largest
mass-balance defect is approximately \(2.7\times10^{-13}\), and no negative
cell values are observed.

\begin{figure}[!t]
	\centering
	\begin{minipage}[t]{0.485\linewidth}
		\centering
		\includegraphics[width=\linewidth]
		{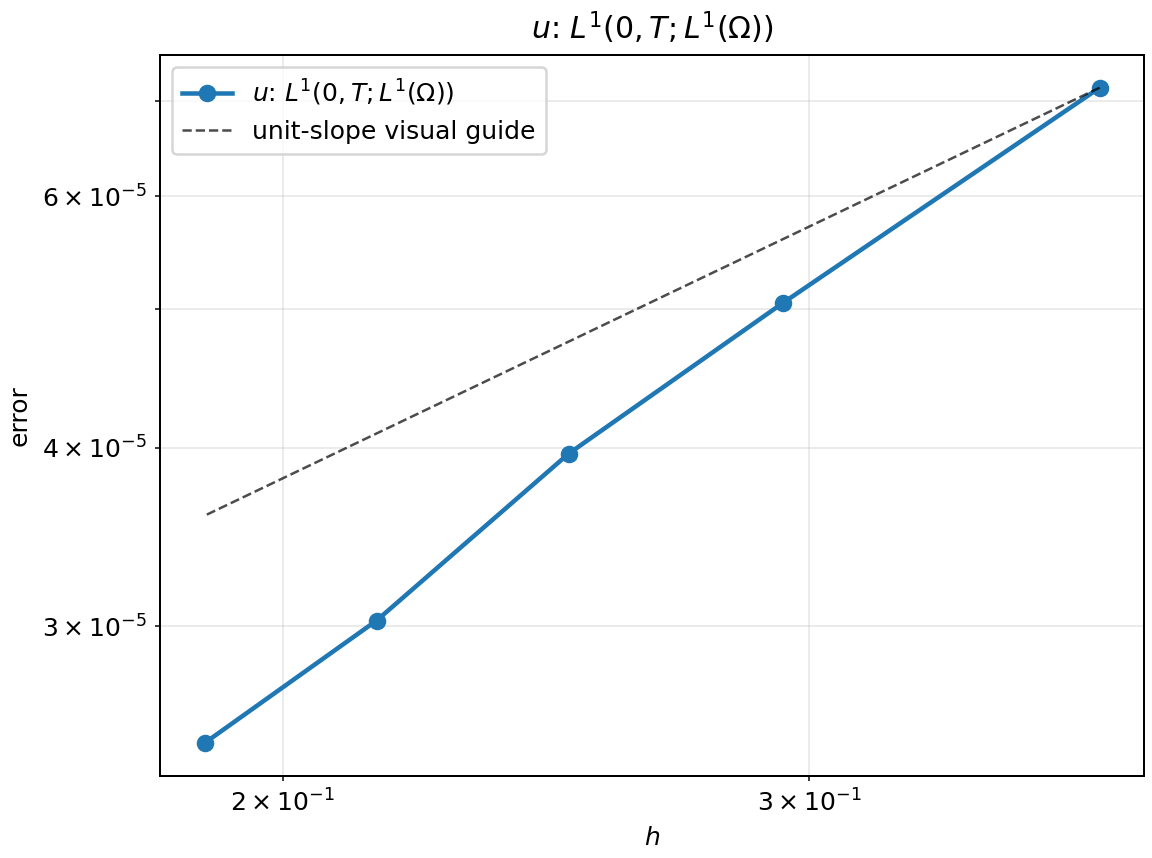}
	\end{minipage}\hfill
	\begin{minipage}[t]{0.485\linewidth}
		\centering
		\includegraphics[width=\linewidth]
		{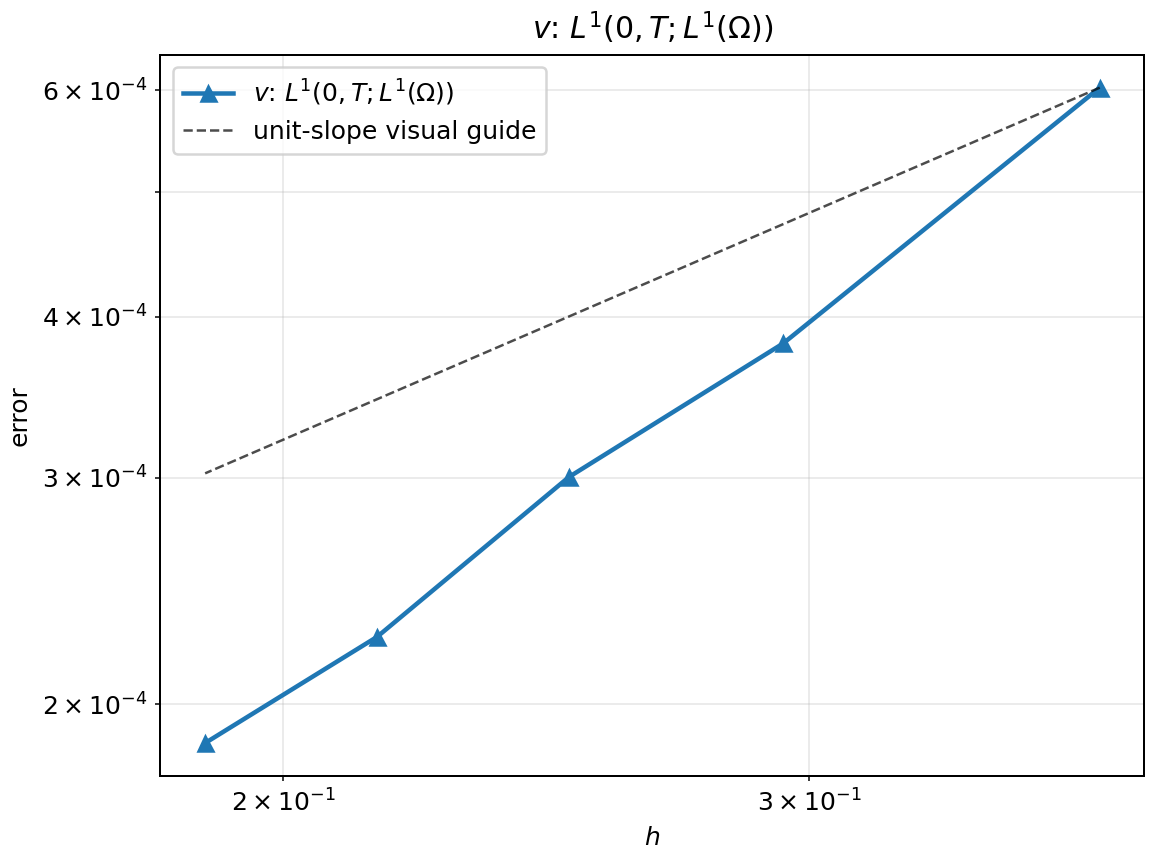}
	\end{minipage}
	
	\medskip
	
	\begin{minipage}[t]{0.485\linewidth}
		\centering
		\includegraphics[width=\linewidth]
		{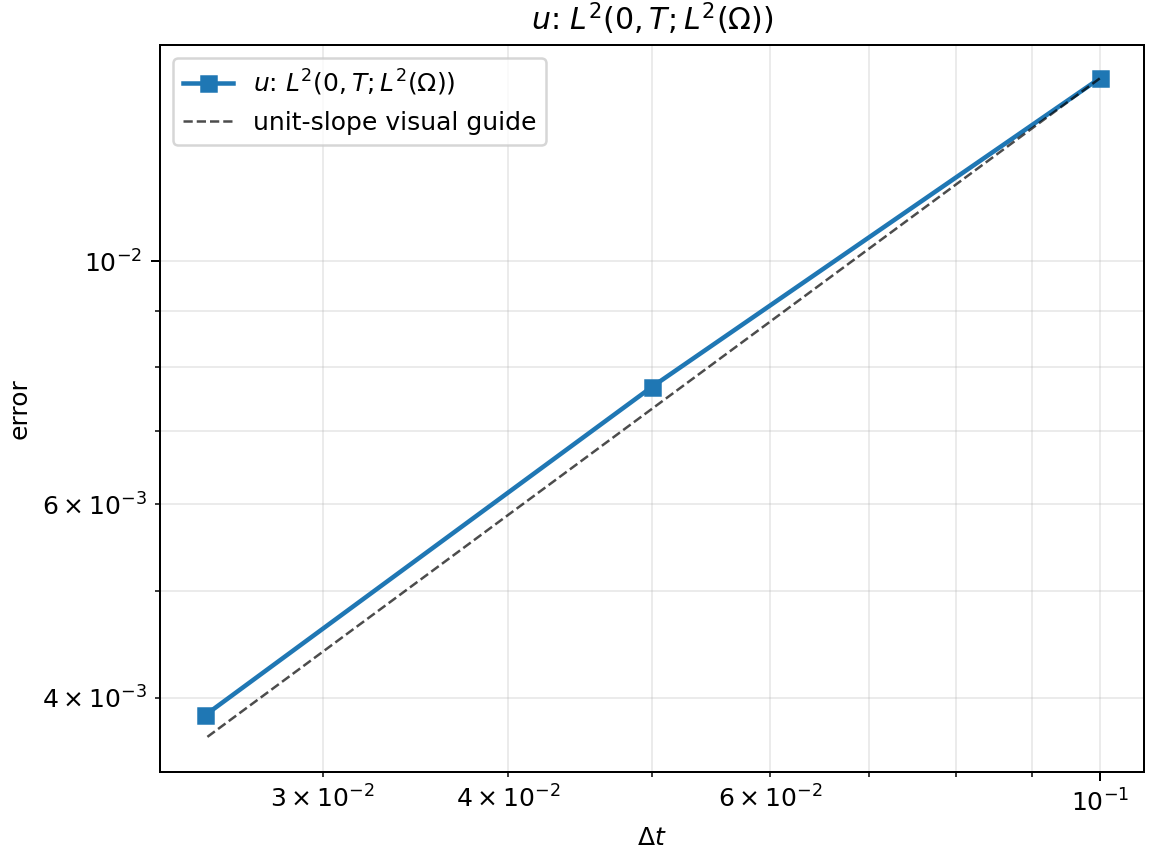}
	\end{minipage}\hfill
	\begin{minipage}[t]{0.485\linewidth}
		\centering
		\includegraphics[width=\linewidth]
		{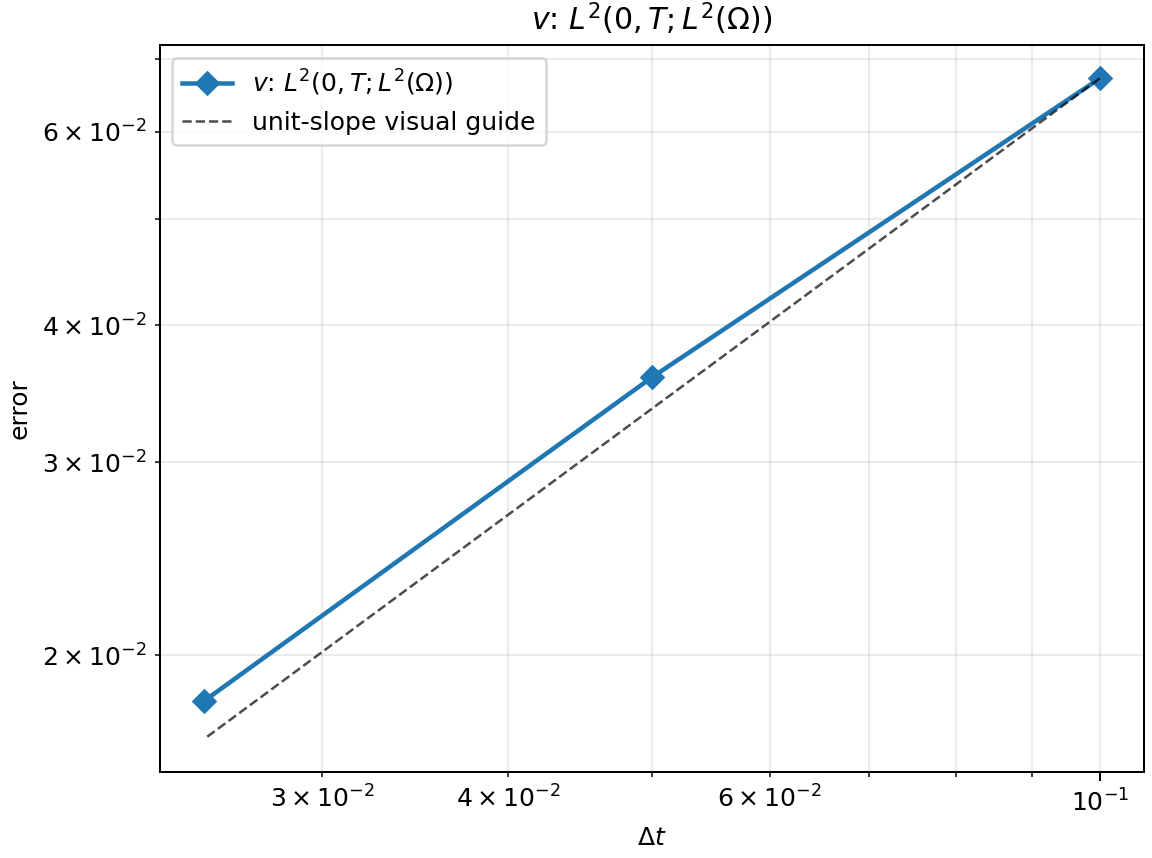}
	\end{minipage}
	
	\caption{Selected manufactured-solution errors. The upper panels show the
		errors along the joint refinement path
		\(\Delta t\simeq0.005h_{\mathcal D}^2\); the lower panels show fixed-mesh
		temporal refinement against the checked fine-time reference. Dashed
		slope-one lines are included only as visual guides.}
	\label{fig:manufactured-convergence}
\end{figure}
	\section{Conclusion and perspectives}
	\label{sec:conclusion}
	
	For each fixed sensing length and on domains admitting a uniformly regular
	orthogonal mesh sequence, we proved subsequential convergence of the
	finite-volume approximations to a nonnegative weak solution of the nonlocal
	microglia--amyloid system with a spatially measure-valued vascular influx.
	The proof combines discrete positivity, mass and energy estimates, global
	truncation estimates for the signal, space--time compactness, and consistency
	of the diffusive and complete upwind chemotactic fluxes
	\cite{EymardGallouetHerbin2000,DroniouGallouetHerbin2003,
		BoccardoGallouet1989}.
	
	The numerical test reports decreasing errors along a joint space--time
	refinement path toward an exact smooth full-edge-flux manufactured solution
	of the augmented system. On a fixed mesh, the temporal errors are compatible
	with first-order backward Euler over the tested time-step range. The reported
	nonlinear algebraic residuals and mass-balance defects are small. No separate
	asymptotic spatial order or quantitative error estimate is inferred from
	these computations.
	
	The manufactured boundary measure is absolutely continuous with respect to
	boundary arclength. Singular boundary measures are covered by the convergence
	analysis but are not benchmarked numerically here. No uniqueness result,
	quantitative error estimate, or nonlocal-to-local limit is established.
	Future work includes error analysis, numerical tests for singular vascular
	inputs, parameter identification, uncertainty analysis, and optimal control.
	
	\section*{Statements and Declarations}
	
	\paragraph{Funding.}
	This work was supported by Cadi Ayyad University (UCA)
	through the Young Researchers of UCA research project call,
	second edition (2026), under the project entitled
	``Mathematical and numerical modelling of Alzheimer's disease
	progression: integration of artificial-intelligence techniques
	into Keller--Segel dynamics.''
	
	\paragraph{Competing interests.}
	The author has no relevant financial or non-financial interests
	to disclose.
	
	\bibliographystyle{plain}
	\bibliography{references}
	
	\appendix
	
	\section{Auxiliary reconstruction and consistency results}
	\label{app:auxiliary-results}
	
	Throughout this appendix,
	\((\mathcal D_m,\Delta t_m)_{m\geq1}\) is a uniformly regular sequence
	of admissible orthogonal discretizations, with
	\[
	h_m:=h_{\mathcal D_m}\to0,\qquad
	\Delta t_m\to0,\qquad
	I_m^k:=(t_m^k,t_m^{k+1}].
	\]
	For \(e=K|L\in\mathcal E_{m,\mathrm{int}}\), set
	\[
	\begin{gathered}
		\nu_e:=\frac{x_L-x_K}{d_{KL}},\qquad
		\mathbb M_m:=d\,\nu_e\otimes\nu_e\ \text{on }D_e,\qquad
		\mathbb M_m:=0\ \text{on boundary subdiamonds},\\
		\|\mathbb M_m\|_{L^\infty(\Omega)}\leq d,\qquad
		\mathbb M_m^\top=\mathbb M_m,\qquad
		\mathbb M_m^2=d\mathbb M_m .
	\end{gathered}
	\]
	We also write
	\[
	\delta_ez_m^{k+1}:=z_{m,L}^{k+1}-z_{m,K}^{k+1},
	\qquad
	\mathcal A_m(w_m,z_m):=
	\sum_k\Delta t_m
	\sum_{e=K|L\in\mathcal E_{m,\mathrm{int}}}
	\tau_e\,\delta_ew_m^{k+1}\delta_ez_m^{k+1}.
	\]

	The auxiliary results are organized according to their role:
	Proposition~\ref{prop:appendix-geometric-consistency} treats the mesh
	geometry and diffusion terms,
	Proposition~\ref{prop:appendix-spatial-reconstructions} the spatial
	reconstructions, and
	Proposition~\ref{prop:appendix-elementary-reconstructions} the temporal,
	data, and boundary-measure reconstructions.
	
	\begin{proposition}[Geometric and diffusive consistency]
		\label{prop:appendix-geometric-consistency}
		\label{lem:appendix-mesh-tensor}
		\label{lem:appendix-point-sampled-tests}
		\label{cor:appendix-diffusion-consistency}
		The following properties hold.
		
		\begin{enumerate}
			\item[(i)]
			\[
			\mathbb M_m\rightharpoonup^\ast I
			\quad\text{in }
			L^\infty(\Omega)^{d\times d}
			\ \text{and }L^\infty(Q_T)^{d\times d}.
			\]
			
			\item[(ii)]
			For \(\zeta\in C^1([0,T];C^2(\overline\Omega))\), set
			\(\zeta_{m,K}^{k+1}:=\zeta(x_K,t_m^{k+1})\). Then
			\begin{equation}
				\label{eq:appendix-test-gradient}
				\nabla_{\mathcal D_m}\zeta_m
				=\mathbb M_m\nabla\zeta+r_m^\zeta,
				\qquad
				\|r_m^\zeta\|_{L^\infty(Q_T)}
				\leq C_\zeta(h_m+\Delta t_m).
			\end{equation}
			
			\item[(iii)]
			If \(1<p<\infty\) and
			\[
			\nabla_{\mathcal D_m}w_m\rightharpoonup\nabla w
			\quad\text{in }L^p(Q_T)^d,
			\]
			then
			\[
			\mathcal A_m(w_m,\zeta_m)
			\longrightarrow
			\int_{Q_T}\nabla w\cdot\nabla\zeta\,dx\,dt.
			\]
		\end{enumerate}
	\end{proposition}
	
	\begin{proof}
		For \(i=1,\ldots,d\), let \(X_{m,K}^{(i)}:=(x_K)_i\). Then
		\[
		\|X_m^{(i)}-x_i\|_{L^\infty(\Omega)}\leq h_m,
		\qquad
		\nabla_{\mathcal D_m}X_m^{(i)}=\mathbb M_me_i.
		\]
		For \(\Phi\in C_c^1(\Omega)^d\), comparison of diamond and face averages
		gives
		\[
		\int_\Omega\mathbb M_me_i\cdot\Phi\,dx
		=
		\sum_{e=K|L}(x_L-x_K)_i
		\int_e\Phi\cdot\nu_e\,dS+\mathcal R_m,
		\qquad
		|\mathcal R_m|
		\leq Ch_m\|\nabla\Phi\|_{L^\infty(\Omega)}.
		\]
		For \(m\) sufficiently large, \(\Phi\) vanishes on the exterior faces,
		and therefore
		\[
		\begin{aligned}
			\sum_{e=K|L}(x_L-x_K)_i\int_e\Phi\cdot\nu_e\,dS
			&=-\sum_K(x_K)_i\int_K\nabla\cdot\Phi\,dx=-\int_\Omega X_m^{(i)}\nabla\cdot\Phi\,dx.
		\end{aligned}
		\]
		Consequently,
		\[
		\int_\Omega\mathbb M_me_i\cdot\Phi\,dx
		\longrightarrow
		-\int_\Omega x_i\nabla\cdot\Phi\,dx
		=
		\int_\Omega e_i\cdot\Phi\,dx.
		\]
		The uniform \(L^\infty\)-bound and density prove (i) column by column;
		the space--time statement follows because \(\mathbb M_m\) is
		time-independent.
		
		For \((x,t)\in D_e\times I_m^k\),
		\[
		\begin{aligned}
			\nabla_{\mathcal D_m}\zeta_m(x,t)
			&=
			d\nu_e\int_0^1
			\nabla\zeta\bigl(x_K+s(x_L-x_K),t_m^{k+1}\bigr)
			\cdot\nu_e\,ds\\
			&=
			\mathbb M_m\nabla\zeta(x,t)+r_m^\zeta(x,t).
		\end{aligned}
		\]
		The spatial and temporal arguments differ from \((x,t)\) by at most
		\(Ch_m\) and \(\Delta t_m\), proving
		\eqref{eq:appendix-test-gradient}.
		
		Finally, using
		\(\mathbb M_m\nabla_{\mathcal D_m}w_m
		=d\nabla_{\mathcal D_m}w_m\),
		\[
		\begin{aligned}
			\mathcal A_m(w_m,\zeta_m)
			&=
			\frac1d\int_{Q_T}
			\nabla_{\mathcal D_m}w_m\cdot
			\nabla_{\mathcal D_m}\zeta_m\,dx\,dt\\
			&=
			\int_{Q_T}
			\nabla_{\mathcal D_m}w_m\cdot\nabla\zeta\,dx\,dt
			+\frac1d\int_{Q_T}
			\nabla_{\mathcal D_m}w_m\cdot r_m^\zeta\,dx\,dt\\
			&\longrightarrow
			\int_{Q_T}\nabla w\cdot\nabla\zeta\,dx\,dt.
		\end{aligned}
		\]
	\end{proof}
	
	\begin{proposition}[Consistency of the spatial reconstructions]
		\label{prop:appendix-spatial-reconstructions}
		\label{lem:appendix-nonlocal-potential}
		\label{lem:appendix-diamond-average}
		Let \(v_m\geq0\) satisfy
		\[
		\sup_m\|v_m\|_{L^\infty(0,T;L^1(\Omega))}<\infty.
		\]
		
		\begin{enumerate}
			\item[(i)]
			For \(t\in I_m^k\), define
			\[
			\begin{aligned}
				Z_m(x)&:=\sum_{M\in\mathcal T_m}|M|K_\sigma(x,x_M),&
				N_m(x,t)&:=
				\sum_{M\in\mathcal T_m}
				|M|K_\sigma(x,x_M)v_{m,M}^{k+1},&
				\widetilde c_m&:=\frac{N_m}{Z_m}.
			\end{aligned}
			\]
			Then \(\widetilde c_m(x_K,t)
			=K_{\sigma,\mathcal D_m}[v_m^{k+1}]_K\), and
			\begin{align}
				\|\widetilde c_m-\mathcal K_\sigma[v_m]\|_
				{L^\infty(0,T;W^{1,\infty}(\Omega))}
				&\leq C_{\sigma,T}h_m,
				\label{eq:appendix-potential-quadrature}\\
				\sup_m\|\widetilde c_m\|_
				{L^\infty(0,T;W^{2,\infty}(\Omega))}
				&\leq C_{\sigma,T}.
				\label{eq:appendix-potential-W2}
			\end{align}
			If \(v_m\to v\) in \(L^1(Q_T)\) and
			\(v\in L^\infty(0,T;L^1(\Omega))\), then
			\begin{align}
				\widetilde c_m
				&\to\mathcal K_\sigma[v]
				&&\text{in }L^2(0,T;W^{1,\infty}(\Omega)),
				\label{eq:appendix-potential-strong}\\
				\nabla_{\mathcal D_m}c_m
				&=\mathbb M_m\nabla\widetilde c_m+r_m^c,
				&
				\|r_m^c\|_{L^\infty(Q_T)}
				&\leq C_{\sigma,T}h_m,
				\label{eq:appendix-potential-gradient}
			\end{align}
			where \(c_{m,K}^{k+1}:=\widetilde c_m(x_K,t)\).
			
			\item[(ii)]
			For \(t\in I_m^k\), define
			\[
			\overline u_m:=
			\frac{u_{m,K}^{k+1}+u_{m,L}^{k+1}}2
			\ \text{on }D_e,\qquad
			\overline u_m:=u_{m,K}^{k+1}
			\ \text{on boundary subdiamonds}.
			\]
			Then
			\begin{equation}
				\label{eq:appendix-diamond-average-estimate}
				\|\overline u_m-u_m\|_{L^2(Q_T)}^2
				\leq
				Ch_m^2\sum_k\Delta t_m
				|u_m^{k+1}|_{1,\mathcal D_m}^2.
			\end{equation}
			Hence, if \(u_m\to u\) in \(L^2(Q_T)\) and
			\(\sup_m\sum_k\Delta t_m|u_m^{k+1}|_{1,\mathcal D_m}^2<\infty\),
			then \(\overline u_m\to u\) in \(L^2(Q_T)\).
		\end{enumerate}
	\end{proposition}
	
	\begin{proof}
		For every multi-index \(|\alpha|\leq1\),
		\[
		\begin{aligned}
			\|\partial_x^\alpha
			(N_m-\mathcal K_\sigma[v_m])\|_{L^\infty(\Omega)}
			&\leq C_\sigma h_m\|v_m(t)\|_{L^1(\Omega)},&
			\|Z_m-1\|_{W^{1,\infty}(\Omega)}
			&\leq C_\sigma h_m,\\
			Z_m&\geq\kappa_\sigma|\Omega|,&
			\|N_m(t)\|_{W^{2,\infty}(\Omega)}
			+\|Z_m\|_{W^{2,\infty}(\Omega)}
			&\leq C_{\sigma,T}.
		\end{aligned}
		\]
		Indeed, the first estimate follows from
		\[
		\partial_x^\alpha N_m-\partial_x^\alpha\mathcal K_\sigma[v_m]
		=
		\sum_Mv_{m,M}^{k+1}
		\int_M
		\bigl[
		\partial_x^\alpha K_\sigma(x,x_M)
		-\partial_x^\alpha K_\sigma(x,y)
		\bigr]\,dy.
		\]
		The quotient rule proves
		\eqref{eq:appendix-potential-quadrature} and
		\eqref{eq:appendix-potential-W2}.
		
		Moreover,
		\[
		\|\mathcal K_\sigma[v_m](t)-\mathcal K_\sigma[v](t)\|_
		{W^{1,\infty}(\Omega)}
		\leq C_\sigma\|v_m(t)-v(t)\|_{L^1(\Omega)}.
		\]
		The left-hand side converges in \(L^1(0,T)\) and is uniformly bounded
		in \(L^\infty(0,T)\), hence converges in \(L^2(0,T)\). Together with
		\eqref{eq:appendix-potential-quadrature}, this proves
		\eqref{eq:appendix-potential-strong}. The spatial Taylor argument from
		Proposition~\ref{prop:appendix-geometric-consistency} gives
		\eqref{eq:appendix-potential-gradient}.
		
		Finally,
		\[
		\begin{aligned}
			\|\overline u_m-u_m\|_{L^2(Q_T)}^2
			=
			\frac1{4d}\sum_k\Delta t_m
			\sum_{e=K|L}|e|d_{KL}
			|\delta_eu_m^{k+1}|^2\leq
			Ch_m^2\sum_k\Delta t_m
			|u_m^{k+1}|_{1,\mathcal D_m}^2,
		\end{aligned}
		\]
		because \(d_{KL}\leq2h_m\). The conclusion follows by the triangle
		inequality.
	\end{proof}
	
	\begin{proposition}[Consistency of time, data, and boundary reconstructions]
		\label{prop:appendix-elementary-reconstructions}
		\label{lem:appendix-left-reconstruction}
		\label{lem:appendix-cell-time-average}
		\label{lem:appendix-boundary-measure}
		The following properties hold.
		
		\begin{enumerate}
			\item[(i)]
			For \(1\leq p<\infty\), define
			\[
			w_m(t):=w_m^{k+1},\qquad
			w_m^-(t):=w_m^k,\qquad t\in I_m^k.
			\]
			If \(w_m\to w\) in \(L^p(Q_T)\) and
			\(\sup_m\|w_m^0\|_{L^p(\Omega)}<\infty\), then
			\(
			w_m^-\to w\quad\text{in }L^p(Q_T).
			\)
			
			\item[(ii)]
			With
			\[
			(\Pi_mf)_K^{k+1}
			:=
			\frac1{|K|\Delta t_m}
			\int_{I_m^k}\int_Kf(x,t)\,dx\,dt,
			\]
			one has, for every \(f\in L^p(Q_T)\),
			\[
			\|\Pi_mf\|_{L^p(Q_T)}\leq\|f\|_{L^p(Q_T)},
			\qquad
			\Pi_mf\to f\quad\text{in }L^p(Q_T).
			\]
			In particular, if \(u_m\to u\) in \(L^2(Q_T)\) and
			\(u_{\mathrm r}\in L^\infty(Q_T)\), then
			\[
			\Pi_m[(u_{\mathrm r}-u_m)^+]
			\to(u_{\mathrm r}-u)^+
			\quad\text{in }L^2(Q_T).
			\]
			
			\item[(iii)]
			For \(\psi\in\mathscr T_T\), define
			\[
			\psi_m^\partial(\xi,t):=
			\psi(x_{K(e)},t_m^{k+1})
			\quad\text{on }\widehat e\times I_m^k,
			\qquad
			\psi_m^\partial:=\psi
			\quad\text{on }\Sigma\times[0,T].
			\]
			Then
			\begin{align}
				\|\psi_m^\partial-\psi\|_
				{L^\infty(\Gamma_{\mathrm v}\times[0,T])}
				&\leq C_\psi(h_m+\Delta t_m),
				\label{eq:appendix-boundary-uniform}\\
				\sum_k\sum_{e\in\mathcal E_{m,\mathrm v}}
				\overline\mu(\widehat e\times I_m^k)
				\psi(x_{K(e)},t_m^{k+1})
				&\longrightarrow
				\int_{\Gamma_{\mathrm v}\times(0,T]}
				\psi\,d\overline\mu.
				\label{eq:appendix-boundary-limit}
			\end{align}
		\end{enumerate}
	\end{proposition}
	
	\begin{proof}
		Since \(w_m^-(t)=w_m(t-\Delta t_m)\) for
		\(t>\Delta t_m\),
		\[
		\begin{aligned}
			\|w_m^--w\|_{L^p(Q_T)}^p
			\leq{}&
			C_p\|w_m-w\|_{L^p(Q_T)}^p
			+C_p\|w(\cdot+\Delta t_m)-w\|_{L^p(Q_T)}^p\\
			&+C_p\Delta t_m\|w_m^0\|_{L^p(\Omega)}^p
			+C_p\int_0^{\Delta t_m}\|w(t)\|_{L^p(\Omega)}^p\,dt
			\longrightarrow0.
		\end{aligned}
		\]
		
		Jensen's inequality gives
		\(\|\Pi_mf\|_{L^p}\leq\|f\|_{L^p}\), while density of smooth functions
		gives \(\Pi_mf\to f\). With \(G:=(u_{\mathrm r}-u)^+\),
		\[
		\begin{aligned}
			\|\Pi_m[(u_{\mathrm r}-u_m)^+]-G\|_{L^2}
			&\leq
			\|\Pi_m[(u_{\mathrm r}-u_m)^+-(u_{\mathrm r}-u)^+]\|_{L^2}
			+\|\Pi_mG-G\|_{L^2}\\
			&\leq
			\|u_m-u\|_{L^2}+\|\Pi_mG-G\|_{L^2}
			\longrightarrow0.
		\end{aligned}
		\]
		
		Finally, for
		\((\xi,t)\in\widehat e\times I_m^k\),
		\[
		|x_{K(e)}-\xi|\leq h_m,\qquad
		|t_m^{k+1}-t|\leq\Delta t_m,
		\]
		which proves \eqref{eq:appendix-boundary-uniform}. Since
		\((\widehat e)_e\) partitions
		\(\Gamma_{\mathrm v}\setminus\Sigma\) and
		\(\overline\mu(\Sigma\times[0,T])=0\),
		\[
		\begin{aligned}
			\sum_{k,e}
			\overline\mu(\widehat e\times I_m^k)
			\psi(x_{K(e)},t_m^{k+1})
			&=
			\int_{\Gamma_{\mathrm v}\times(0,T]}
			\psi_m^\partial\,d\overline\mu,\\
			\left|\int
			(\psi_m^\partial-\psi)\,d\overline\mu\right|
			&\leq
			\|\psi_m^\partial-\psi\|_{L^\infty}
			\overline\mu(\Gamma_{\mathrm v}\times[0,T])
			\longrightarrow0.
		\end{aligned}
		\]
	\end{proof}
	
\end{document}